\documentclass[11pt]{article}
\usepackage{amsfonts, amsmath, amssymb, amsthm, color, dsfont, epsfig, epstopdf, enumerate}
\usepackage{framed, graphicx, hyperref, listings, mathtools, multicol, multirow, pdfpages, physics}
\usepackage{amsopn, authblk, setspace, subfig}

\usepackage{algorithm, algpseudocode}
\usepackage[table]{xcolor}

\newtheorem{theorem}{Theorem}[section]
\newtheorem{proposition}[theorem]{Proposition}
\newtheorem{lemma}[theorem]{Lemma}
\newtheorem{corollary}[theorem]{Corollary}

\newtheorem{definition}[theorem]{Definition}
\newtheorem{remark}[theorem]{Remark}

\DeclareMathOperator{\dir}{{\rm div}}

\newcommand{\1}{\mathds{1}}
\newcommand{\R}{\mathbb{R}}
\newcommand{\abel}{\mathcal{A}}
\newcommand{\abelJ}{\mathcal{J}}
\newcommand{\dist}{{\rm dist}}

\newcommand{\lip}{\mathrm{Lip}}
\newcommand{\normal}{\mathrm{Normal}}
\newcommand{\supp}{{\rm supp}}
\numberwithin{equation}{section}
\numberwithin{algorithm}{section}
\numberwithin{figure}{section}
\numberwithin{table}{section}
\numberwithin{theorem}{section}

\title{Stability and Error Estimates of BV Solutions \\ to the Abel Inverse Problem}
\author{Linan Zhang}
\author{Hayden Schaeffer}
\affil{Department of Mathematical Sciences, Carnegie Mellon University, Pittsburgh, PA 15213. (\text{linanz@andrew.cmu.edu}, { }\text{schaeffer@cmu.edu})}
\date{June, 2018}

\begin{document}
\maketitle


\begin{abstract}
Reconstructing images from ill-posed inverse problems often utilizes total variation regularization in order to recover discontinuities in the data while also removing noise and other artifacts. Total variation regularization has been successful in recovering images for (noisy) Abel transformed data, where object boundaries and data support will lead to sharp edges in the reconstructed image. In this work, we analyze the behavior of $BV$ solutions to the Abel inverse problem, deriving \textit{a priori} estimates on the recovery.  In particular, we provide $L^2$-stability bounds on $BV$ solutions to the Abel inverse problem. These bounds yield error estimates on images reconstructed from a proposed total variation regularized minimization problem. 
\end{abstract}


\section{Introduction}
\label{sec:intro}

The Abel integral equation arises in a variety of fields, including medical imaging, astronomy, geophysics, and electron microscopy \cite{Gorenflo91, Poularikas00, Bracewell03, Epstein08}.  The Abel equation is an essential tool in many aspects of science, since internal structures (such as density, composition, velocity profiles, \textit{etc.}) of an object can be reconstructed from just their line-of-sight projections, with the assumption that the structures are axisymmetric or nearly-axisymmetric. The reconstruction, \textit{i.e.} the inversion of the Abel integral equation, is an ill-posed problem due to a lack of smoothness in the data and solution (typically in the form of discontinuities along object or material boundaries) and the presence of random additive noise.

In practice, several analytical and numerical approaches were proposed to (essentially) deconvolve the integral equation. In \cite{Kalal88, Smith88}, the authors used the Abel-Fourier-Hankel cycle \cite{Bracewell03} to solve the inverse problem. This is based on the projection-slice theorem and uses a special relationship between the Abel and Fourier transforms. In \cite{Dribinski02, Ma07}, the authors used a basis-set expansion (either on the solution or the projection) to solve the inverse problem, where the action of the Abel transform is analytically calculated on each basis functions and the coefficients are solved computationally. It is noted that the linear system for the coefficients required Tikhonov regularization to avoid ill-conditioning. Other approaches include the Cormack inversion \cite{Cormack63, Cormack64, Solomon84} and the onion-peeling method \cite{Dasch92, Pretzier92}. Without additional regularization, these methods tend to amplify noise due to the ill-conditioning of the discrete inverse problem. This is a result of the ill-posedness of the continuous problem \cite{Gorenflo91,Poularikas00}. In addition, these methods are often not suitable for discontinuous data.

To handle discontinuities and noise, it is natural to consider restricting solutions to functions of bounded variation. This is done by adding a total variation penalty on the inverse problem. Total variation (TV) regularization is an essential part of many inversion methods in image processing, originating from the ROF model \cite{Rudin92} for denoising, and now popular in many models, including, for example: compressive sensing and medical imaging \cite{candes2005signal, candes2006robust,  lustig2007sparse}, video processing \cite{li2009user, li2013efficient, yang2013mixing, schaeffer2015space, schaeffer2015real}, and cartoon-texture decomposition \cite{meyer2001oscillating, osher2003image, schaeffer2013low}. Let $\abel$ be the Abel transform, $f$ be the line-of-sight projection of some unknown non-negative axisymmetric function $u$, and $\Omega\subset\R^2$ be the compact support of $f$. In \cite{Asaki05,Asaki06}, the authors presented the following TV regularized minimization problem:
\begin{align}
\min_{u\in BV(U)}\quad & \|u\|_{TV(U)} + \dfrac{\lambda}{2}\|\abel u-f\|_{L^2(\Omega)}^2, \label{Asaki05 eq3}
\end{align}
where $U$ is the rotation of $\Omega$ about the $z$-axis, and thus the TV semi-norm (for axisymmetric functions) is defined as follows:
\begin{align*}
\|u\|_{TV(U)} :=  \sup\left\{ \iint_\Omega u(r,z) \dir(r\phi(r,z))\dd{r}\dd{z}: \phi\in C_c^1(\Omega;\R^2), \, \|\phi\|_{L^\infty(\Omega)}\le1 \right\},
\end{align*}
where the divergence is calculated with respect to the variables $(r,z)$. In \cite{Asaki05}, it was shown that if $f\in L^2(\Omega)$, then Problem \eqref{Asaki05 eq3} has a unique global minimizer in $BV(U)$. In addition, numerical results on discontinuous functions showed that solving Equation \eqref{Asaki05 eq3} yields better results versus unregularized inversion or $H^1$ regularized inversion. 
In \cite{Chartrand10}, the variational model in Equation \eqref{Asaki05 eq3} was modified by adding the box constraint:
\begin{equation*} 
\begin{split}
\min_{u\in BV(U)}\quad & \|u\|_{TV(U)}+ \dfrac{\lambda}{2}\|\abel u-f\|_{L^2(\Omega)}^2 \\
\text{subject to} \quad & u\in[a,b] \text{ on } U,
\end{split}
\end{equation*}
which was noted to have better denoising results than the model without the constraint. In \cite{Abraham08}, the authors proved existence and/or uniqueness results for various models, including the binary minimization problem:
\begin{equation*} 
\begin{split}
\min_{u\in BV(\Omega)}\quad & \|u\|_{TV(\Omega)}  + \dfrac{\lambda}{2}\|\abel u-f\|_{L^2(\Omega)}^2 \\
\text{subject to} \quad & u\in\{0,1\} \text{ a.e. on } \Omega,
\end{split}
\end{equation*}
where the TV semi-norm over $\Omega$ is defined as follows:
\begin{align*}
\|u\|_{TV(\Omega)} :=  \sup\left\{ \iint_\Omega u(r,z) \dir\phi(r,z)\dd{r}\dd{z}: \phi\in C_c^1(\Omega;\R^2), \, \|\phi\|_{L^\infty(\Omega)}\le1 \right\}.
\end{align*}
Additionally, a higher-order TV regularization for the Abel inverse problem was proposed in \cite{Chan15},  where the regularization term is the sum of the TV semi-norm and the $L^1$ norm of the Laplacian. Numerical experiments showed that the addition of the $L^1$ norm of the Laplacian helps to recover piecewise smooth data as opposed to piecewise constant data typically recovered by TV regularized inversion. In each of these variational models, the main regularization involves the TV semi-norm, thus resulting in $BV$ solutions.  In addition, the data is fit with respect to the $L^2$ norm. Therefore, one would expect to control the $L^2$-error between an approximation and the true image by the TV semi-norm of the solution and the $L^2$ norm of the given data. 

\subsection{Contributions of this work}
 
Motivated by the various total variation regularized Abel inversion models, we provide analytic and numerical results on the behavior of $BV$ solutions of the Abel inverse problem. Since many of the related variational models involve the TV semi-norm and an $L^2$ data-fit, we derive error bounds using these terms.

In particular, we provide \textit{a priori} $L^2(U)$-stability bounds of $BV$ solutions, and from these estimates, we derive an $L^2(U)$-error estimate when regularizing the inverse problem by $\|u\|_{TV(\Omega)}$ (the total variation of the solution in cylindrical coordinates). Amongst the choices of TV-regularizers for the Abel inverse problem, our analysis shows that $\|u\|_{TV(\Omega)}$ naturally arises as an error control term. The motivation for $L^2$-error bounds comes from the fact that many recovery results are measured by the root-mean squared error. Several numerical examples verify that our variational model and error bound yields satisfactory results. 

Our derivation also yields an $L^1(U)$-stability bound of $BV$ solutions, which agrees (and simplifies) the $L^1$-bound for 1D problems found in \cite{Gorenflo91}. Note that the $L^2$-bounds in \cite{Gorenflo91} do not apply to $BV$ solutions, and extensions of the results in \cite{Gorenflo91} can be shown to be suboptimal for the problem considered in this work. Therefore, we have derived different and new bounds which are applicable to the problem considered here. In addition, we present the first error results for $BV$ solutions that hold for 1D and 2D data (with 2D and 3D solutions respectively). 

\subsection{Overview}

This paper is organized as follows. In Section \ref{sec: model}, we discuss the Abel inverse problem and a TV regularized model. In Section \ref{sec: theory}, we derive stability and error bounds for $BV$ solutions in both two and three dimensions. In Section \ref{sec: example},  some examples are shown, which verify the theoretical estimates.


\section{Inverse Problem and Variational Method}
\label{sec: model}

In this section, we define the Abel integral operator as well as a variational model used for the inversion. Although there are several choices for the variational model, the particular one used here naturally occurs within our analysis. 


\begin{definition} \label{def: A transform}
Let $u:\R^2\to\R$ be an axisymmetric function. The Abel transform of $u$ is defined as:
\begin{align}
\abel u(x) := 2 \int_x^\infty \dfrac{u(r)r}{\sqrt{r^2-x^2}}\dd{r}, \quad x\in\R^+. \label{eq: def abel 2d} 
\end{align}
If $u:\R^3\to\R$ is an axisymmetric function, then the Abel transform of $u$ is defined as:
\begin{align*}
\abel u(x,z) := 2 \int_x^\infty \dfrac{u(r,z)r}{\sqrt{r^2-x^2}}\dd{r}, \quad (x,z)\in\R^+\times\R. 
\end{align*}
\end{definition}


\begin{figure}[b!]
  \centering
  \includegraphics[scale=0.2]{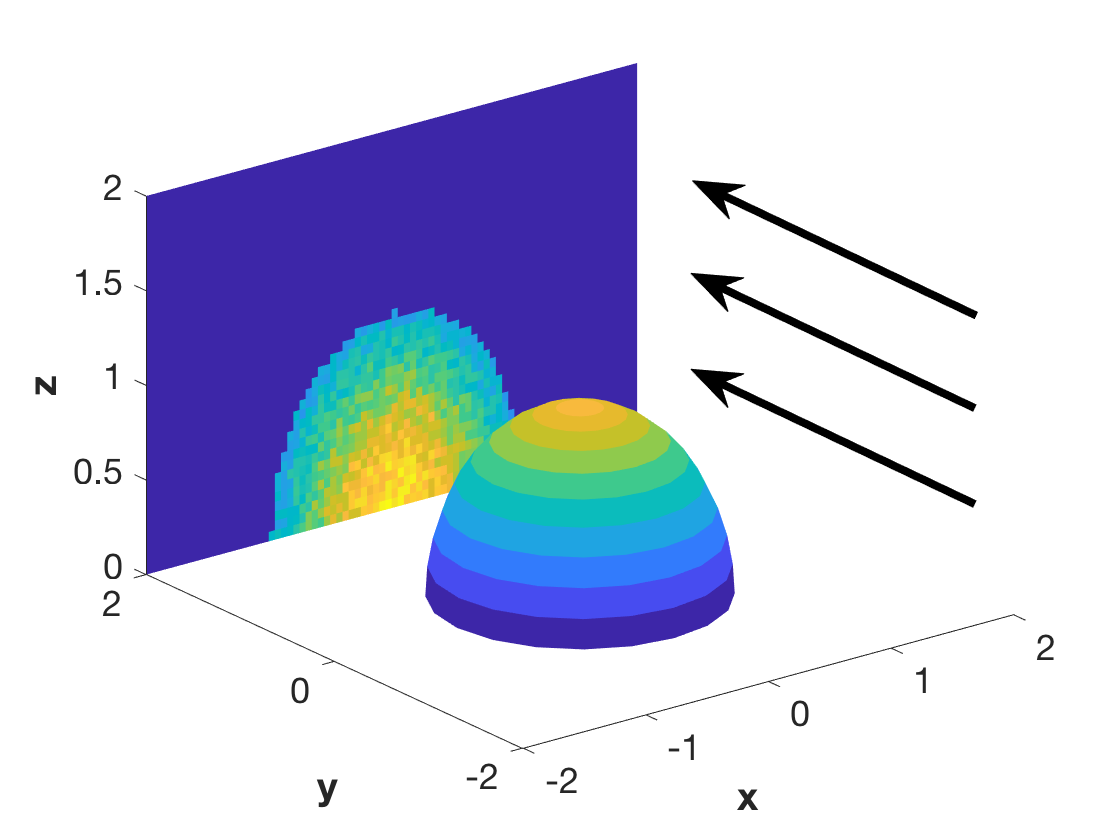}
  \caption{Graphical description of the Abel transform. The axisymmetric object (with compact support) is integrated along lines parallel to the $y$-axis, and a line-of-sight projection is obtained. The line of sight is illustrated by the arrows. The projection on the $(x,z)$-plane contains noise due to measurement error. The inverse problem is to reconstruct the 3D object from the 2D projection.}
  \label{fig:ex0_algorithm}
\end{figure}


Let $u$ be the unknown density of an axisymmetric object and suppose that we are given data $f$, which is the line-of-sight projection of $u$ onto a co-dimension one domain (see Figure~\ref{fig:ex0_algorithm}). In this work, we always assume that $u$ is non-negative. To reconstruct $u$ from $f$, one can solve the following linear system:
\begin{align}
f = \abel u. \label{eq: f=Au model}\tag{\mbox{P$_{\abel}$}}
\end{align}
This is known as the {\it Abel integral equation}. We also assume that the data and solution are compactly supported within the domain of interest. The solution may have discontinuities along object boundaries or may have jumps at the support boundary, thus $BV$ solutions should be expected.

In Theorems \ref{thm: PA uniqueness 2d} and \ref{thm: PA uniqueness 3d}, it is shown that if the function $f$ is of bounded variation, then Problem \eqref{eq: f=Au model} has a unique solution in $L^1$. However, to ensure the solutions are in $BV$, we consider the TV regularized problem:
\begin{align}
\min_{u\in BV(\Omega)}\quad & \|u\|_{TV(\Omega)}+ \dfrac{\lambda}{2}\|\abel u-f\|_{L^2(\Omega)}^2, \label{eq: TV model}\tag{\mbox{P$_{\rm TV}$}}
\end{align}
where $f$ is the given data with compact support $\Omega\subset\R^d$, $d\in\{1,2\}$, and $\|u\|_{TV(\Omega)}$ is the total variation of $u$ in $\Omega$ and is defined by:
\begin{align}
\|u\|_{TV(\Omega)} :=  \sup\left\{ \int_\Omega u(r) \dir\phi(r)\dd{r}: \phi\in C_c^1(\Omega;\R^2), \, \|\phi\|_{L^\infty(\Omega)}\le1 \right\} \label{eq: def TV norm 2d}
\end{align}
if $d=1$, or
\begin{align}
\|u\|_{TV(\Omega)} :=  \sup\left\{ \iint_\Omega u(r,z) \dir\phi(r,z)\dd{r}\dd{z}: \phi\in C_c^1(\Omega;\R^2), \, \|\phi\|_{L^\infty(\Omega)}\le1 \right\} \label{eq: def TV norm 3d}
\end{align}
if $d=2$. Let $u^*$ be the minimizer to Problem \eqref{eq: TV model}. Informally, the regularization term $\|u\|_{TV(\Omega)}$ ensures that $u^*\in BV(\Omega)$ and allows for discontinuities in the radial profile, and the loss term $\|\abel u-f\|_{L^2(\Omega)}^2$ enforces that $\abel u^* \approx f$ in the presence of Gaussian noise.

\begin{remark} \label{rem: difference from Asaki}
Problem \eqref{eq: TV model} differs from Problem \eqref{Asaki05 eq3} in that the regularization term in Problem \eqref{eq: TV model} is the total variation of $u$ with respect to the cylindrical coordinates, while the regularization term in Problem \eqref{Asaki05 eq3} is the total variation of $u$ with respect to the Cartesian coordinates. See Proposition \ref{prop: TV norm cartesian}.
\end{remark}

Without loss of generality, assume that the support of the data is contained in $[0,1)\subset\R$, if the data is 1D, or $[0,1)\times[-1,1]\subset\R^2$, if the data is 2D. The results developed in the subsequent sections can be generalized to any bounded domain in the same form by introducing an additional constant depending only on the size of the domain. We denote $\Omega = [0,1]\times[-1,1]\subset\R^2$, $\Omega_h = [h,1]\times[-1,1]\subset\R^2$ where $h\in(0,1/2]$, and $U=\{(x,y)\in\R^2:x^2+y^2\le1\}\times[-1,1]\subset\R^3$ in order to simplify the notations.


\section{Stability and Error Estimates}
\label{sec: theory}

In this section, we analyze the stability of $BV$ solutions to Problem \eqref{eq: f=Au model} and estimate the error of $BV$ solutions to Problem \eqref{eq: TV model}. In particular, we show that given a minimizer of Problem \eqref{eq: TV model}, we can control the $L^2$ norm of the solution in terms of a fixed multiple of the $L^2$ norm of the data. This provides a quantitative error bound that is not common for this type of model. The stability estimates derived in this work are:
\begin{align*}
\|u\|_{L^2(B(0,1))} &\le C\|u\|_{TV(0,1)}^{1/2}\|f\|_{L^2(0,1)}^{1/2}
\end{align*}
if the data is 1D (solutions are 2D), and
\begin{align*}
\|u\|_{L^2(U)} &\le C\|u\|_{L^\infty(\Omega)}^{1/3}\|u\|_{TV(\Omega)}^{1/3}\|f\|_{L^2(\Omega)}^{1/3}
\end{align*}
if the data is 2D (solutions are 3D). Note that the right-hand side of the inequalities are expressed in terms of norms and semi-norms over co-dimension one regions. These bounds yield an error estimate for Problem \eqref{eq: TV model}.

The motivation for deriving $L^2$-bounds is to provide variance control over the $BV$ regularized least-squares solution. In practical applications, the fit of the recovered solution is measured by the root-mean squared error. Therefore, it is natural to look for a theoretical bound on the $L^2$-error. In related variational models, the recovered images are assumed to be in $BV$, and the data-fit is measured by the $L^2$ norm. Thus, our error bounds can be controlled by the (easily available) TV semi-norm of the solution and the $L^2$ norm of the given data.  

We first introduce the following integral transform for functions, which is known as the Weyl fractional integral of order $1/2$ and is closely related to the Abel transform \cite{Poularikas00}.


\begin{definition} \label{def: J transform}
Let $v$ be a scalar-valued function defined on $\R^+$. The $\abelJ$-transform of $v$ is defined as:
\begin{align}
\abelJ v(x) := \dfrac{1}{\sqrt{\pi}} \int_x^\infty \dfrac{v(r)}{\sqrt{r-x}}\dd{r}, \quad x\in\R^+. \label{eq: def abelJ 2d}
\end{align}
If $v$ is a scalar-valued function defined on $\R^+\times\R$, the $\abelJ$-transform of $v$ is defined as:
\begin{align}
\abelJ v(x,z) := \dfrac{1}{\sqrt{\pi}} \int_x^\infty \dfrac{v(r,z)}{\sqrt{r-x}}\dd{r}, \quad (x,z)\in\R^+\times\R. \label{eq: def abelJ 3d}
\end{align}
\end{definition}


Let $v$ be an unknown function defined on $\R^+$ or $\R^+\times\R$, and suppose that we are given the $\abelJ$-transform of $v$, denoted by $g$. An analogy to the Abel integral equation \eqref{eq: f=Au model} is the following:
\begin{align}
g = \abelJ v. \label{eq: g=Jv model}\tag{\mbox{P$_{\abelJ}$}}
\end{align}
The connection between Problems \eqref{eq: f=Au model} and \eqref{eq: g=Jv model} is shown in Sections \ref{sec: theory2d} and \ref{sec: theory3d}. In this work, we assume that $v$ is also non-negative.

In the following subsections, we will focus on $L^2$-stability estimates of $BV$ solutions to Problems \eqref{eq: f=Au model} and \eqref{eq: g=Jv model}, and for the sake of completeness, we will provide the corresponding $L^1$-stability estimates in Appendix \ref{sec: L1 bound}.


\subsection{$L^2$-Stability Estimates for $BV$ Solutions in 2D}
\label{sec: theory2d}

Before analyzing the stability of $BV$ solutions to Problem \eqref{eq: f=Au model}, we discuss the existence and uniqueness of solutions to the inverse problem over different spaces of functions as well as the relationship between the Abel transform and the $\abelJ$-transform. 

Assume that $u:\R^2\to\R$ is an axisymmetric function which is compactly supported within the ball $B(0,1)\subset\R^2$, and that $v:\R^+\to\R$ is defined as $v(r^2) = u(r)$. Changing variables shows that:
\begin{align}
\abel u(x) &= 2 \int_x^1 \dfrac{u(r)r}{\sqrt{r^2-x^2}}\dd{r} = \int_{x^2}^1 \dfrac{u(\sqrt{r})}{\sqrt{r-x^2}}\dd{r} = \sqrt{\pi}\abelJ v(x^2), \quad x\in[0,1]. \label{eq: A-J connection 2d}
\end{align}
The analysis of Problem \eqref{eq: g=Jv model} can thus be related back to the Abel inverse problem through a simple change of variables. The following theorem states the existence and uniqueness of solutions to Problem \eqref{eq: g=Jv model}. 


\begin{theorem} \label{thm: PJ uniqueness 2d}
{\rm (restated from \cite{Gorenflo91})}
Problem \eqref{eq: g=Jv model} has a unique solution in $L^1(0,1)$, which is given by:
\begin{align}
v(r) = -\dfrac{1}{\sqrt{\pi}}\int_r^1 \dfrac{\dd{g(x)}}{\sqrt{x-r}}, \label{eq: PJ solution 2d}
\end{align}
provided that the function $g$ is of bounded variation, $0\le g(0)<\infty$, and $\supp(g)\subset[0,1)$. Here the integral is in the Lebesgue-Stieltjes sense. 
\end{theorem}


The proof of Theorem \ref{thm: PJ uniqueness 2d} appears in Appendix \ref{sec: appendix}. An analogy of Theorem \ref{thm: PJ uniqueness 2d} for Problem \eqref{eq: f=Au model} can be derived immediately from Equations \eqref{eq: A-J connection 2d} and \eqref{eq: PJ solution 2d}.


\begin{theorem}\label{thm: PA uniqueness 2d}
Problem \eqref{eq: f=Au model} has a unique solution in $L^1(0,1)$, which is given by:
\begin{align*}
u(r) = -\dfrac{1}{\pi}\int_r^1 \dfrac{\dd{f(x)}}{\sqrt{x^2-r^2}}, 
\end{align*}
provided that the function $f$ is of bounded variation, $0\le f(0)<\infty$, and $\supp(f)\subset[0,1)$.
\end{theorem}


Theorems \ref{thm: PJ uniqueness 2d} and \ref{thm: PA uniqueness 2d} guarantee only that the solution exists in $L^1$. If it is known {\it a priori} that the solution is in $BV$, then it can be shown that the data is H{\"o}lder continuous.


\begin{theorem} \label{thm: J bv continuity 2d}
{\rm (restated from \cite{Bergounioux16})}
The operator $\abelJ$ defined by Equation \eqref{eq: def abelJ 2d} is a continuous operator from $BV(0,1)$ into $C^{0,1/2}(0,1)$.
\end{theorem}


Here $C^{0,1/2}$ denotes the space of all functions on $[0,1]$ which satisfy the H{\"o}lder condition of order $1/2$. This result is a direct consequence of Proposition \ref{prop: J linf continuity 2d} and Poincar\'{e}'s inequality in 1D. An analogy of Theorem \ref{thm: J bv continuity 2d} can be derived immediately from Equation \eqref{eq: A-J connection 2d}.


\begin{corollary} \label{cor: A bv continuity 2d}
The operator $\abel$ defined by Equation \eqref{eq: def abel 2d} is a continuous operator from $BV(0,1)$ into $C^{0,1/2}(0,1)$.
\end{corollary}


The statement above shows that, in practice, one may need to regularize the Abel inverse problem in order to ensure that the solutions are in $BV$. 

We now focus on the stability of $BV$ solutions to Problems \eqref{eq: f=Au model} and \eqref{eq: g=Jv model}. The following lemma provides two basic estimates for the ``running average,'' $v_h$, of a function $v$ defined on the interval $[0,1]\subset\R$, where $h\in(0, 1/2]$. The introduction of the auxiliary function $v_h$ helps us to bound $v$ in $L^2(h,1)$, with a bound that is a function of $h$. Then one can minimize the bound in $h$ over $(0, 1/2]$ to obtain a bound of $v$ in $L^2(0,1)$.


\begin{lemma} \label{lem: basic bounds 2d}
Let $v\in W^{1,1}(0,1)$. Let $h\in(0, 1/2]$ and define $v_h:[h,1]\to\R$ by:
\begin{align}
v_h(x) := \dfrac{1}{h} \int_{x-h}^x v(y)\dd{y}. \label{eq: running average 2d}
\end{align}
Then the following two estimates hold:
\begin{align}
\|v-v_h\|_{L^2(h,1)} &\le 3^{-1/2}h^{1/2}\|v'\|_{L^1(0,1)}, \label{eq: basic bound 2d} \\
\|v_h\|_{L^\infty(h,1)} &\le \|v\|_{L^\infty(0,1)}. \label{eq: vh inf bound 2d}
\end{align}
\end{lemma}


\begin{proof}
We first verify that  $v_h - v$ can be written as a convolution for $x\in[h,1]$:
\begin{align}
v_h(x) - v(x) &= \int_0^1v'(y)\left(\dfrac{x-y}{h}-1\right)\1_{[0,h]}(x-y)\dd{y}. \label{eq: basic bound eq1 2d}
\end{align}
Since $v\in W^{1,1}(0,1)$, it is absolutely continuous, and by the fundamental theorem of calculus, we have:
\begin{align*}
\int_0^1v'(y)\left(\dfrac{x-y}{h}-1\right)\1_{[0,h]}(x-y)\dd{y} &= \int_{x-h}^xv'(y)\left(\dfrac{x-y}{h}\right)\dd{y} - \int_{x-h}^xv'(y)\dd{y} \\
&= -v(x-h) +  \dfrac{1}{h} \int_{x-h}^x v(y)\dd{y} - \left(v(x) - v(x-h)\right) \\
&= v_h(x) - v(x).
\end{align*}
It can be seen from Equation \eqref{eq: basic bound eq1 2d} that $v_h - v = K*g$ on $[h,1]$, where
\begin{align*}
K(x) := \left(\dfrac{x}{h} -1\right)\1_{[0,h]}(x), \quad g(x) := v'(x)\mathds{1}_{[0,1]}(x),
\end{align*}
and we have extended the functions to $\R$. Applying Young's inequality for convolutions, we obtain:
\begin{align*}
\|v-v_h\|_{L^2(h,1)} = \|K*g\|_{L^2(h,1)} \leq \|K*g\|_{L^2(\R)} \le \|K\|_{L^2(\R)}\|g\|_{L^1(\R)} = 3^{-1/2}h^{1/2}\|v'\|_{L^1(0,1)},
\end{align*}
where the last equality can be calculated directly. This shows Equation \eqref{eq: basic bound 2d}.

By Equation \eqref{eq: running average 2d}, for $x\in[h,1]$,
\begin{align*}
|v_h(x)| \le \dfrac{1}{h} \int_{x-h}^x |v(y)|\dd{y} \le \|v\|_{L^\infty(0,1)},
\end{align*}
which shows Equation \eqref{eq: vh inf bound 2d}.
\end{proof}


\begin{remark} \label{rem: basic bounds alternative 2d}
One could obtain an alternative bound for $\|v-v_h\|_{L^2(h,1)}$ by using the $L^p$ embedding theorem and Poincar\'{e}'s inequality in 1D:
\begin{align}
\|v-v_h\|_{L^2(h,1)} &\le \|v\|_{L^2(h,1)} + \|v_h\|_{L^2(h,1)} \le (1-h)^{1/2}\left(\|v\|_{L^\infty(h,1)} + \|v_h\|_{L^\infty(h,1)}\right) \nonumber \\
&\le 2(1-h)^{1/2}\|v\|_{L^\infty(0,1)} \le 2(1-h)^{1/2}\|v'\|_{L^1(0,1)}. \label{eq: basic bound alternative 2d}
\end{align}
Note that $3^{-1/2}h^{1/2}<2(1-h)^{1/2}$ for $h\in(0, 1/2]$, so that the estimate in Equation \eqref{eq: basic bound alternative 2d} is not as tight as the estimate in Equation \eqref{eq: basic bound 2d}. In addition, Equation \eqref{eq: basic bound alternative 2d} could lead to complications in later arguments.
\end{remark}


The following theorem shows a stability estimate for $W^{1,1}$ solutions to Problem \eqref{eq: g=Jv model} in terms of the data itself. 


\begin{theorem} \label{thm: PJ w11 bound 2d}
If $v\in W^{1,1}(0,1)$ with $\supp(v)\subset[0,1)$ and $\abelJ v = g$, we have:
\begin{align}
\|v\|_{L^2(0,1)} \le C\|v'\|_{L^1(0,1)}^{1/2}\|g\|_{L^2(0,1)}^{1/2}, \label{eq: PJ w11 bound 2d}
\end{align}
where $C$ is a constant independent of $v$.
\end{theorem}

\begin{proof}
We first note that for each $x\in[0,1]$, Fubini's theorem implies that
\begin{align}
\abelJ^2 v(x) &= \dfrac{1}{\sqrt{\pi}}\int_x^1\dfrac{\abelJ v(s)}{\sqrt{s-x}} \dd{s} = \dfrac{1}{\pi}\int_x^1 \dfrac{1}{\sqrt{s-x}} \int_s^1 \dfrac{v(y)}{\sqrt{y-s}} \dd{y} \dd{s} \nonumber \\
&= \dfrac{1}{\pi}\int_x^1 v(y) \int_x^y \dfrac{1}{\sqrt{s-x}\sqrt{y-s}} \dd{s} \dd{y} = \int_x^1v(y)\dd{y}, \label{eq: PJ-w11 eq1 2d}
\end{align}
where the last step follows from Equation \eqref{eq: gamma identity}.  This is the key to the argument, specifically, that two applications of $\abelJ$ is the same as integration. Then from Equations \eqref{eq: running average 2d} and \eqref{eq: PJ-w11 eq1 2d}, we have, for $x\in[h,1]$,
\begin{align}
\sqrt{\pi}hv_h(x) &= \sqrt{\pi}\left(\int_{x-h}^1 v(y)\dd{y} - \int_x^1 v(y)\dd{y}\right) = \sqrt{\pi}\left(\abelJ g(x-h) - \abelJ g(x) \right) \nonumber \\
&= \int_{x-h}^1\dfrac{g(y)}{\sqrt{y-(x-h)}} \dd{y} - \int_x^1\dfrac{g(y)}{\sqrt{y-x}} \dd{y} \nonumber \\
&= \int_{x-h}^x\dfrac{g(y)}{\sqrt{y-(x-h)}} \dd{y} + \int_x^1g(y) \left[\dfrac{1}{\sqrt{y-(x-h)}} - \dfrac{1}{\sqrt{y-x}}\right] \dd{y} \nonumber \\
&= \int_0^1 g(y) K_1(x,y) \dd{y} + \int_0^1g(y)K_2(x,y) \dd{y}  \nonumber\\
&=: F_1(x) + F_2(x), \label{eq: PJ-w11 eq2 2d}
\end{align}
where we extend the kernels to the entire domain and define them by:
\begin{align*}
K_1(x,y):= \dfrac{\1_{[0,h]}(x-y)}{\sqrt{h-(x-y)}}, \quad K_2(x,y):= \dfrac{\1_{[0,h]}(y-x)}{\sqrt{y-(x-h)}} - \dfrac{\1_{[0,h]}(y-x)}{\sqrt{y-x}}. 
\end{align*}
Note that the support set $[0,h]$ in $K_2$ is a by-product of the assumption $x\in[h,1]$. Since the kernels are in $L^1$:
\begin{align*}
\int_0^1 |K_1(x, y)|\dd{y} &= \int_0^1 |K_1(x, y)|\dd{x} = 2h^{1/2}, \\
\& \quad \int_0^1 |K_2(x, y)|\dd{y} &= \int_0^1 |K_2(x, y)|\dd{x} = 2\left(2-\sqrt{2}\right)h^{1/2},
\end{align*}
and by, for example, Theorem 6.18 in \cite{Folland99}, we have $L^2$ control over each term in Equation \eqref{eq: PJ-w11 eq2 2d}:
\begin{align}
\|F_1\|_{L^2(h,1)} \le 2h^{1/2} \|g\|_{L^2(0,1)}, \quad \|F_2\|_{L^2(h,1)} \le 2\left(2-\sqrt{2}\right)h^{1/2} \|g\|_{L^2(0,1)}. \label{eq: PJ-w11 eq3 2d}
\end{align}
Combining Equations \eqref{eq: PJ-w11 eq2 2d}-\eqref{eq: PJ-w11 eq3 2d}, we obtain:
\begin{align}
\|v_h\|_{L^2(h,1)} \le 2\left(3-\sqrt{2}\right)\pi^{-1/2}h^{-1/2} \|g\|_{L^2(0,1)}. \label{eq: PJ-w11 eq4 2d}
\end{align}
On the other hand, using the $L^p$ interpolation theorem and Poincar\'{e}'s inequality in 1D, we obtain:
\begin{align}
\|v\|_{L^2(0,h)} &\le h^{1/2}\|v\|_{L^\infty(0,h)} \le h^{1/2}\|v\|_{L^\infty(0,1)} \le h^{1/2}\|v'\|_{L^1(0,1)}. \label{eq: PJ-w11 eq5 2d}
\end{align}
Thus, by the triangle inequality and Equations \eqref{eq: basic bound 2d} and \eqref{eq: PJ-w11 eq4 2d}-\eqref{eq: PJ-w11 eq5 2d}, we have:
\begin{align}
\|v\|_{L^2(0,1)} &\le \|v\|_{L^2(0,h)} + \|v-v_h\|_{L^2(h,1)} + \|v_h\|_{L^2(h,1)} \nonumber \\
&\le \left(1+1/\sqrt{3}\right)h^{1/2}\|v'\|_{L^1(0,1)} + 2\left(3-\sqrt{2}\right)\pi^{-1/2}h^{-1/2}\|g\|_{L^2(0,1)} \nonumber \\
&\le 2\left(1+1/\sqrt{3}\right)h^{1/2}\|v'\|_{L^1(0,1)} + 2\left(3-\sqrt{2}\right)\pi^{-1/2}h^{-1/2}\|g\|_{L^2(0,1)}, \label{eq: PJ-w11 eq6 2d}
\end{align}
where the last step follows from slightly extending the upper bound, since we will optimize Equation \eqref{eq: PJ-w11 eq6 2d} with the constraint $h\in(0, 1/2]$. By direct calculation, the value $h^*$ that minimizes the right-hand side of Equation \eqref{eq: PJ-w11 eq6 2d} is given by:
\begin{align}
h^* = \dfrac{\left(3-\sqrt{2}\right)\|g\|_{L^2(0,1)}}{\sqrt{\pi}\left(1+1/\sqrt{3}\right)\|v'\|_{L^1(0,1)}}. \label{eq: PJ-w11 eq7 2d}
\end{align}
We can verify that the minimizer, which depends on the factor $\|g\|_{L^2(0,1)}\|v'\|_{L^1(0,1)}^{-1}$, satisfies the constraint $h^*\in(0, 1/2]$ as follows. By Equation \eqref{eq: def abelJ 2d} and integrating by parts, we have, for $x\in[0,1]$,
\begin{align}
g(x) = \dfrac{1}{\sqrt{\pi}} \int_x^1 \dfrac{v(y)}{\sqrt{y-x}}\dd{y} &= - \dfrac{2}{\sqrt{\pi}} \int_x^1 v'(y)\sqrt{|x-y|}\dd{y}, \label{eq: PJ-w11 eq8 2d}
\end{align}
since the assumption is that $\supp(v)\subset[0,1)$. To bound $g$ in $L^2(0,1)$ by $\|v'\|_{L^1(0,1)}$, we apply Young's inequality for convolutions with
\begin{align}
K(x) := -\dfrac{2}{\sqrt{\pi}}\sqrt{|x|}\mathds{1}_{[0,1]}(-x), \quad f(x) := v'(x)\mathds{1}_{[0,1]}(x). \label{eq: PJ-w11 kernel g}
\end{align}
Thus $g = K*f$ on $[0,1]$ and, with $K$ and $f$ extended to $\R$,
\begin{align}
\|g\|_{L^2(0,1)} = \|K*f\|_{L^2(0,1)} \le \|K*f\|_{L^2(\R)} \le \|K\|_{L^2(\R)}\|f\|_{L^1(\R)} = \sqrt{2/\pi} \|v'\|_{L^1(0,1)}, \label{eq: PJ-w11 eq9 2d}
\end{align}
where the last step can be calculated directly. Therefore, combining Equations \eqref{eq: PJ-w11 eq7 2d} and \eqref{eq: PJ-w11 eq9 2d} yields:
\begin{align*}
h^* = \dfrac{\left(3-\sqrt{2}\right)\|g\|_{L^2(0,1)}}{\sqrt{\pi}\left(1+1/\sqrt{3}\right)\|v'\|_{L^1(0,1)}} \le \dfrac{3\sqrt{2}-2}{\pi\left(1+1/\sqrt{3}\right)} \le \dfrac{1}{2}. 
\end{align*}
Therefore, by optimizing Equation \eqref{eq: PJ-w11 eq6 2d} over $h$, we obtain:
\begin{align*}
\|v\|_{L^2(0,1)} \le 2\pi^{-1/4}\left(1+1/\sqrt{3}\right)^{1/2}\left(3-\sqrt{2}\right)^{1/2}\|v'\|_{L^1(0,1)}^{1/2}\|g\|_{L^2(0,1)}^{1/2},
\end{align*}
which gives Equation \eqref{eq: PJ w11 bound 2d}. 
\end{proof}


The utility of the stability bound in Equation \eqref{eq: PJ w11 bound 2d} is that the right-hand side of the inequality is of the form of a product-bound depending on the data, which can be calculated in practice. As an example, consider the function $v_k(r) = \chi(kr)$, $k\ge1$, where $\chi$ is the indicator function of the interval $[0,1]$:
\begin{align*}
\chi(x) = 
\begin{cases}
1, \quad \text{if } x\in[0,1], \\
0, \quad \text{otherwise}.
\end{cases}
\end{align*}
This is a prototypical example related to image recovery. Define $g_k:=\abelJ v_k$. It can be shown, by Equation \eqref{eq: def abelJ 2d}, that:
\begin{align*}
g_k(x) = 2\pi^{-1/2}(1/k-x)^{1/2}\chi(kx).
\end{align*}
For $k\ge1$, it is easy to show that $\|v'_k\|_{L^1(0,1)} = 1$, and
\begin{align*}
\|v_k\|_{L^1(0,1)} &= k^{-1}, \quad \|g_k\|_{L^1(0,1)} = \dfrac{4}{3\sqrt{\pi}}k^{-3/2}, \\
\|v_k\|_{L^2(0,1)} &= k^{-1/2}, \quad \|g_k\|_{L^2(0,1)} = \dfrac{\sqrt{2}}{\sqrt{\pi}}k^{-1}. 
\end{align*}
If the bound on $\|v\|_{L^2(0,1)}$ is in the sum-form:
\begin{align}
\|v\|_{L^2(0,1)} &\le C_1\|v'\|_{L^1(0,1)} + C_2\|g\|_{L^2(0,1)} \label{eq: PJ w11 bound 2d alternative}
\end{align}
with constants $C_1$ and $C_2$, then:
\begin{align*}
\lim_{k\to\infty}\|v_k\|_{L^2(0,1)} \le C_1\lim_{k\to\infty}\|v'_k\|_{L^1(0,1)} + C_2\lim_{k\to\infty}\|g_k\|_{L^2(0,1)} = C_1,
\end{align*}
which is suboptimal since $\lim_{k\to\infty}\|v_k\|_{L^2(0,1)}=0$. On the other hand, Equation \eqref{eq: PJ w11 bound 2d} yields $\|v_k\|_{L^2(0,1)} \leq C k^{-1/2}$, which obtains the correct decay rate for this example. Therefore, in terms of applicability, a sum-bound in the form of Equation \eqref{eq: PJ w11 bound 2d alternative} is not as desired as our product-bounds like Equation \eqref{eq: PJ w11 bound 2d}, since the right-hand side of Equation \eqref{eq: PJ w11 bound 2d alternative} is not necessarily made arbitrarily small when $\|g\|_{L^2(0,1)}$ is made arbitrarily small. 

In \cite{Gorenflo91}, the authors proved that if $v\in W^{1,1}$ or if $v\in W^{1,2}$, then
\begin{align}
\|v\|_{L^1(0,1)} &\leq C_1\|v'\|_{L^1(0,1)}^{1/3}\|g\|_{L^1(0,1)}^{2/3} + C_2\|g\|_{L^1(0,1)},\label{eq: RG-SV 8.3.1 L1}\\
\|v\|_{L^2(0,1)} &\leq C_1\|v'\|_{L^2(0,1)}^{1/3}\|g\|_{L^2(0,1)}^{2/3} + C_2\|g\|_{L^2(0,1)},\label{eq: RG-SV 8.3.1 L2}
\end{align}
respectively.  In the $L^2$ case, Theorem \ref{thm: PJ w11 bound 2d} improves the results of Theorem 8.3.1 in \cite{Gorenflo91}; since we provide $L^2$ control for $v\in W^{1,1}$ rather than requiring $v\in W^{1,2}$, Equation \eqref{eq: PJ w11 bound 2d} is more applicable to Problem \eqref{eq: g=Jv model}. One could argue that an alternative $L^2$- bound could be obtained from Equation \eqref{eq: RG-SV 8.3.1 L2} using the $L^p$ interpolation theorem:
\begin{align}
\|v\|_{L^2(0,1)} \le \|v\|_{L^1(0,1)}^{1/2}\|v\|_{L^\infty(0,1)}^{1/2} &\le \left(C_1\|v'\|_{L^1(0,1)}^{1/3}\|g\|_{L^1(0,1)}^{2/3} + C_1\|g\|_{L^1(0,1)}\right)^{1/2} \|v\|_{L^\infty(0,1)}^{1/2} \nonumber \\
&\le \left(C_1\|v'\|_{L^1(0,1)}^{1/3}\|g\|_{L^2(0,1)}^{2/3} + C_2\|g\|_{L^2(0,1)}\right)^{1/2} \|v\|_{L^\infty(0,1)}^{1/2}. \label{eq: RG-SV 8.3.1 L2+BV}
\end{align}
Comparing the various bounds yields (with frequent redefinition of the constants $C$, $C_1$, and $C_2$):
\begin{enumerate}[(i)]
\item an $L^1$-bound derived via our approach (see Appendix \ref{sec: L1 bound}, Equation \eqref{eq: PJ w11 bound 2d l1}):
\begin{align*}
\|v_k\|_{L^1(0,1)} &\leq C\|v_k'\|_{L^1(0,1)}^{1/3}\|g_k\|_{L^1(0,1)}^{2/3} \leq C k^{-1},
\end{align*}
which achieves the correct decay rate for this example, \textit{i.e.} $\|v_k\|_{L^1(0,1)} = k^{-1}$;
\item the $L^1$-bound in \cite{Gorenflo91} (Equation \eqref{eq: RG-SV 8.3.1 L1}):
\begin{align*}
\|v_k\|_{L^1(0,1)} &\leq C_1\|v_k'\|_{L^1(0,1)}^{1/3}\|g_k\|_{L^1(0,1)}^{2/3} + C_2\|g_k\|_{L^1(0,1)} \leq C( k^{-1} + k^{-3/2}),
\end{align*}
which achieves the correct decay rate for this example only when the transient term $k^{-3/2}$ decays;
\item our $L^2$-bound (Equation \eqref{eq: PJ w11 bound 2d}):
\begin{align*}
\|v_k\|^2_{L^2(0,1)} &\le C\|v_k'\|_{L^1(0,1)}\|g_k\|_{L^2(0,1)} \leq C k^{-1},
\end{align*}
which achieves the correct decay rate for this example, \textit{i.e.} $\|v_k\|_{L^2(0,1)} = k^{-1/2}$;
\item an $L^2$-bound from \cite{Gorenflo91} using the interpolation theorem (Equation \eqref{eq: RG-SV 8.3.1 L2+BV}):
\begin{align*}
\|v_k\|_{L^2(0,1)}^2 &\le \left(C_1\|v_k'\|_{L^1(0,1)}^{1/3}\|g_k\|_{L^2(0,1)}^{2/3} + C_2\|g_k\|_{L^2(0,1)}\right)\|v_k\|_{L^\infty(0,1)} \leq C (k^{-2/3} + k^{-1}),
\end{align*}
which does not achieves the correct decay rate for this example.
\end{enumerate}
We see that our $L^1$ and $L^2$ bounds achieve the correct decay rate for this example, and thus are tight in some sense. The error bounds in \cite{Gorenflo91} are suboptimal in $L^2$ and contains transient terms in $L^1$.


We now extend the result in Theorem \ref{thm: PJ w11 bound 2d} to obtain an $L^2$-stability estimate for $BV$ solutions to Problem \eqref{eq: g=Jv model} via a density argument.


\begin{theorem} \label{thm: PJ bv bound 2d}
If $v\in BV(0,1)$ with $\supp(v)\subset[0,1)$ and $\abelJ v = g$, we have:
\begin{align}
\|v\|_{L^2(0,1)} \le C\|v\|_{TV(0,1)}^{1/2}\|g\|_{L^2(0,1)}^{1/2}, \label{eq: PJ bv bound 2d}
\end{align}
where $C$ is a constant independent of $v$.
\end{theorem}
\begin{proof}
By the smooth approximation theorem for $BV$ functions, there exists a sequence of functions $\{v_k\}_{k=1}^\infty\subset W^{1,1}(0,1)\cap C^\infty(0,1)=BV(0,1)\cap C^\infty(0,1)$ such that: 
\begin{subequations} \label{eq: PJ-bv eq1 2d}
\begin{align}
&\|v_k-v\|_{L^1(0,1)}\to0 \quad {\rm as}\,\,k\to\infty, \label{eq: PJ-bv eq1a 2d} \\
& v_k\to v \quad {\rm a.e.\,\,as}\,\,k\to\infty, \label{eq: PJ-bv eq1b 2d} \\
{\rm and}\quad &\|v_k\|_{TV(0,1)}\to\|v\|_{TV(0,1)} \quad {\rm as}\,\,k\to\infty. \label{eq: PJ-bv eq1c 2d}
\end{align}
\end{subequations}
Define $g_k:=\abelJ v_k$. By Theorem \ref{thm: PJ w11 bound 2d},
\begin{align}
\|v_k\|_{L^2(0,1)} \le C\|v_k'\|_{L^1(0,1)}^{1/2}\|g_k\|_{L^2(0,1)}^{1/2}, \label{eq: PJ-bv eq2 2d}
\end{align}
where $C$ is a constant independent of the choice of the approximating sequence. Since $\{v_k\}_{k=1}^\infty\subset C^1(0,1)$, condition \eqref{eq: PJ-bv eq1c 2d} implies that:
\begin{align}
\|v_k'\|_{L^1(0,1)}\to\|v\|_{TV(0,1)} \quad {\rm as}\,\,k\to\infty. \label{eq: PJ-bv eq3 2d}
\end{align}
We now show that 
\begin{align}
\|g_k\|_{L^2(0,1)}\to\|g\|_{L^2(0,1)} \quad {\rm as}\,\,k\to\infty\label{eq: PJ-bv eq4 2d}
\end{align}
by proving $\|g_k-g\|_{L^2(0,1)}\to0$ as $k\to\infty$. Choosing $p=2$ and $\epsilon=1/2$ in Theorem \ref{thm: J lp continuity 2d} so that $s=2$, and applying the $L^p$ interpolation theorem, we have:
\begin{align}
\|g_k-g\|_{L^2(0,1)} &= \|\abelJ(v_k-v)\|_{L^2(0,1)} \le \dfrac{2}{\sqrt{\pi}}\|v_k-v\|_{L^2(0,1)} \nonumber \\
& \le \dfrac{2}{\sqrt{\pi}}\|v_k-v\|_{L^1(0,1)}^{1/2}\|v_k-v\|_{L^\infty(0,1)}^{1/2}. \label{eq: PJ-bv eq5 2d}
\end{align}
By Poincar\'{e}'s inequality in 1D: 
\begin{align}
\|v_k-v\|_{L^\infty(0,1)} &\le \|v_k\|_{L^\infty(0,1)} + \|v\|_{L^\infty(0,1)} \le 2\max\{\|v_k\|_{L^\infty(0,1)},\|v\|_{L^\infty(0,1)}\} \nonumber \\
&\le 2\max\{\|v_k\|_{TV(0,1)},\|v\|_{TV(0,1)}\} \le 4\|v\|_{TV(0,1)}, \label{eq: PJ-bv eq6 2d}
\end{align}
where the last inequality holds by condition \eqref{eq: PJ-bv eq1c 2d} for all $k$ sufficiently large. Thus, Equations \eqref{eq: PJ-bv eq5 2d}-\eqref{eq: PJ-bv eq6 2d} together with condition \eqref{eq: PJ-bv eq1a 2d} imply that:
\begin{align*}
\|g_k-g\|_{L^2(0,1)} \le \dfrac{4}{\sqrt{\pi}}\|v\|_{TV(0,1)}^{1/2}\|v_k-v\|_{L^1(0,1)}^{1/2} \to 0
\end{align*}
as $k\to\infty$, which yields Equation \eqref{eq: PJ-bv eq4 2d}. Therefore, by Equations \eqref{eq: PJ-bv eq2 2d}-\eqref{eq: PJ-bv eq4 2d}:
\begin{align*}
\|v\|_{L^2(0,1)} \le \liminf_{k\to\infty}\|v_k\|_{L^2(0,1)} \le C\lim_{k\to\infty}\|v_k'\|_{L^1(0,1)}^{1/2}\|g_k\|_{L^2(0,1)}^{1/2} = C\|v\|_{TV(0,1)}^{1/2}\|g\|_{L^2(0,1)}^{1/2},
\end{align*}
where the first step follows from condition \eqref{eq: PJ-bv eq1b 2d} and Fatou's Lemma. 
\end{proof}


Using the same density argument, one arrives at the following theorem from Equations \eqref{eq: PJ bv bound 2d} and \eqref{eq: r-bound to xy-bound 2d}, which provides an $L^2$-stability estimate for $BV$ solutions to Problem \eqref{eq: f=Au model}.


\begin{theorem} \label{thm: PA bv bound 2d}
Let $u:B(0,1)\subset\R^2\to\R$ be an axisymmetric function such that, as a function of $r$, $u\in BV(0,1)$ and $\supp(u)\subset[0,1)$. If $\abel u = f$, we have:
\begin{align*}
\|u\|_{L^2(B(0,1))} \le C\|u\|_{TV(0,1)}^{1/2}\|f\|_{L^2(0,1)}^{1/2}, 
\end{align*}
where $C$ is a constant independent of $u$,
\begin{align*}
\|u\|_{L^2(B(0,1))} &:= \left(\iint_{B(0,1)} |u(x,y)|^2\dd{x}\dd{y}\right)^{1/2},
\end{align*}
and $\|u\|_{TV(0,1)}$ is defined by Equation \eqref{eq: def TV norm 2d}.
\end{theorem}

This inequality controls the solution in the entire domain by information on its line-of-sight projections.


\subsection{$L^2$-Stability Estimates for $BV$ Solutions in 3D}
\label{sec: theory3d}

In this subsection, the symbol $D$ refers to the weak derivative of a multi-variable function, and $D_1$ is the weak partial with respect to the first component.

We follow the same organization as in the previous subsection. Assume that $u:\R^3\to\R$ is an axisymmetric function which is compactly supported in the cylinder $U\subset\R^3$, and that $v:\R^+\times\R\to\R$ is the function such that $v(r^2,z) = u(r,z)$. Analogous to Equation \eqref{eq: A-J connection 2d}, the following equation holds: 
\begin{align*}
\abel u(x,z) &= \sqrt{\pi}\abelJ v(x^2,z), \quad (x,z)\in\Omega. 
\end{align*}

The following two theorems state the existence and uniqueness of a solution to problems \eqref{eq: g=Jv model} and \eqref{eq: f=Au model}, respectively, which extend Theorems \ref{thm: PJ uniqueness 2d} and \ref{thm: PA uniqueness 2d} to the case where one solves for 3D axisymmetric solutions given 2D line-of-sight projections.


\begin{theorem} \label{thm: PJ uniqueness 3d}
Problem \eqref{eq: g=Jv model} has a unique solution in $L^1(\Omega)$ provided that the function $g$ is of bounded variation, $0\le g(0,z)<\infty$ for $z\in[-1,1]$, and $\supp(g)\subset[0,1) \times [-1,1]$. For each $r\in[0,1]$ and almost every $z\in[-1,1]$, the solution $v$ is given by:
\begin{align}
v(r,z) = -\dfrac{1}{\sqrt{\pi}}\int_r^1 \dfrac{\dd{g^x}}{\sqrt{x-r}}, \label{eq: PJ solution 3d}
\end{align}
where $g^x$ is a Radon measure such that:
\begin{align*}
\int_0^1 \phi(x)\dd{g^x} = -\int_0^1\phi'(x)g(x,z)\dd{x}
\end{align*} 
for all $\phi\in C^1(0,1)$.
\end{theorem}

\begin{proof}
Let $g$ be as assumed. Since $g$ is of bounded variation, by, for example, Theorem 2 on page 220 in \cite{Evans92}, $g(\cdot,z)$ is of bounded variation for almost every $z\in[-1,1]$. Fix $z\in[-1,1]$ such that $g(\cdot,z)$ is of bounded variation. Then by Theorem \ref{thm: PJ uniqueness 2d}, the solution $v(\cdot,z)$ to the problem $\abelJ v(\cdot,z)=g(\cdot,z)$ is in $L^1(0,1)$ and is uniquely given by Equation \eqref{eq: PJ solution 3d}. In particular, Equation \eqref{eq: PJ uniqueness eq1 2d} in the proof of Theorem \ref{thm: PJ uniqueness 2d} implies that $v\in L^1(\Omega)$.
\end{proof}


\begin{theorem} \label{thm: PA uniqueness 3d}
Problem \eqref{eq: f=Au model} has a unique solution in $L^1(\Omega)$ provided that the function $f$ is of bounded variation, $0\le f(0,z)<\infty$ for $z\in[-1,1]$, and $\supp(f)\subset[0,1) \times [-1,1]$. For each $r\in[0,1]$ and almost every $z\in[-1,1]$, the solution $u$ is given by:
\begin{align*}
u(r,z) = -\dfrac{1}{\pi}\int_r^1 \dfrac{\dd{f^x}}{\sqrt{x^2-r^2}}, 
\end{align*}
where $f^x$ is a Radon measure such that:
\begin{align*}
\int_0^1 \phi(x)\dd{f^x} = -\int_0^1\phi'(x)f(x,z)\dd{x}
\end{align*} 
for all $\phi\in C^1(0,1)$.
\end{theorem}


\begin{remark} \label{rem: bv continuity 3d}
Theorem \ref{thm: J bv continuity 2d} and Corollary \ref{cor: A bv continuity 2d} can be extended to 3D axisymmetric solutions given 2D data, but one can only provide $1/2$-H{\"o}lder continuity along almost every line of integration. Unfortunately, global conditions are not guaranteed; a counterexample can be constructed as follows. Let $v$ be a scalar-valued function defined on $\Omega$ such that:
\begin{align*}
v(r,z) = 
\begin{cases}
1, \quad \text{if } z\in S, \\
0, \quad \text{if } z\notin S,
\end{cases}
\end{align*}
where $S\subset[0,1]$ is a non-measurable set. Then $v(\cdot,z)\in BV(0,1)$ for each fixed $z$, but $v$ is not in $BV(\Omega)$. This motivates the use of global total variation penalty, $\|v\|_{TV}$, rather than a penalty along each line, $\int_{-1}^1\|v(\cdot,z)\|_{TV(0,1)}\dd{z}$.
\end{remark}


The following lemma provides two basic estimates for the running average of a function defined on $\Omega$ along lines parallel to an axis. The auxiliary function $v_h$ defined below plays a similar role to the one defined in Equation \eqref{eq: running average 2d}.


\begin{lemma} \label{lem: basic bounds 3d}
Let $v\in W^{1,1}(\Omega)\cap L^\infty(\Omega)$. Let $h\in(0, 1/2]$ and define $v_h:\Omega_h\to\R$ by:
\begin{align}
v_h(x,z) := \dfrac{1}{h} \int_{x-h}^x v(y,z)\dd{y}. \label{eq: running average 3d}
\end{align}
Then the following two estimates hold:
\begin{align}
\|v-v_h\|_{L^2(\Omega_h)} &\le \left(4/3\right)^{1/4}h^{1/4}\|v\|_{L^\infty(\Omega)}^{1/2} \|Dv\|_{L^1(\Omega)}^{1/2}, \label{eq: basic bound 3d} \\
\|v_h\|_{L^\infty(\Omega_h)} &\le \|v\|_{L^\infty(\Omega)}. \label{eq: vh inf bound 3d}
\end{align}
\end{lemma}

\begin{proof}
By, for example, Theorem 10.35 in \cite{Leoni09}, for almost every $z\in[-1,1]$, $v(\cdot,z)$ is absolutely continuous, so that $v_h$ is well-defined. Replacing $v(\cdot)$ by $v(\cdot,z)$ in the proof of Lemma \ref{lem: basic bounds 2d}, one can obtain Equation \eqref{eq: vh inf bound 3d} from Equation \eqref{eq: vh inf bound 2d} and the following estimate from Equation~\eqref{eq: basic bound 2d}:
\begin{align}
\|v(\cdot,z)-v_h(\cdot,z)\|_{L^2(h,1)} \le 3^{-1/2}h^{1/2}\|D_1v(\cdot,z)\|_{L^1(0,1)}, \quad \text{a.e.}\ z\in[-1,1]. \label{eq: basic bound eq1 3d}
\end{align}
To obtain Equation \eqref{eq: basic bound 3d} from Equation \eqref{eq: basic bound eq1 3d}, we first apply the $L^p$ embedding theorem:
\begin{align}
\|v-v_h\|_{L^1(\Omega_h)} &= \int_{-1}^1\|v(\cdot,z)-v_h(\cdot,z)\|_{L^1(h,1)}\dd{z} \le \int_{-1}^1\|v(\cdot,z)-v_h(\cdot,z)\|_{L^2(h,1)}\dd{z}\nonumber \\
&\le 3^{-1/2}h^{1/2}\int_{-1}^1\|D_1v(\cdot,z)\|_{L^1(0,1)}\dd{z} = 3^{-1/2}h^{1/2} \|D_1v\|_{L^1(\Omega)}, \label{eq: basic bound eq2 3d}
\end{align}
which gives a bound that is a function of $h$. Then we apply the $L^p$ interpolation theorem:
\begin{flalign*}
&& \|v-v_h\|_{L^2(\Omega_h)}^2 &\le \|v-v_h\|_{L^\infty(\Omega_h)}\|v-v_h\|_{L^1(\Omega_h)} && \\
&& &\le (\|v\|_{L^\infty(\Omega_h)}+\|v_h\|_{L^\infty(\Omega_h)})\|v-v_h\|_{L^1(\Omega_h)} && \\
&& &\le 2\|v\|_{L^\infty(\Omega)}\|v-v_h\|_{L^1(\Omega_h)} && \text{(by Eq. \eqref{eq: vh inf bound 3d})} \\
&& &\le \frac{2}{\sqrt{3}}h^{1/2} \|v\|_{L^\infty(\Omega)}\|D_1v\|_{L^1(\Omega)} && \text{(by Eq. \eqref{eq: basic bound eq2 3d})} \\
&& &\le \frac{2}{\sqrt{3}}h^{1/2} \|v\|_{L^\infty(\Omega)}\|Dv\|_{L^1(\Omega)}, && 
\end{flalign*}
where the last step follows from Lemma \ref{lem: partial total variation}. 
\end{proof}


\begin{remark} \label{rem: basic bounds alternative 3d}
One may be able to avoid the introduction of $\|v\|_{L^\infty(\Omega)}$ into the error bound for $\|v-v_h\|_{L^2(\Omega_h)}$ via an argument similar to the one in Remark \ref{rem: basic bounds alternative 2d}. However, similar issue may arise since there might not be an interior minimizer in $h$ when estimating $\|v\|_{L^2(\Omega)}$.
\end{remark}


The following theorem shows a stability estimate for $W^{1,1}\cap L^\infty$ solutions to Problem \eqref{eq: g=Jv model}. The additional $L^\infty$ condition is reasonable given that we are recovering images. 

\begin{theorem} \label{thm: PJ w11 bound 3d}
If $v\in W^{1,1}(\Omega)\cap L^\infty(\Omega)$ with $\supp(v)\subset[0,1)\times[-1,1]$ and $\abelJ v = g$, we have:
\begin{align}
\|v\|_{L^2(\Omega)} &\le C\|v\|_{L^\infty(\Omega)}^{1/3}\|Dv\|_{L^1(\Omega)}^{1/3}\|g\|_{L^2(\Omega)}^{1/3}, \label{eq: PJ w11 bound 3d}
\end{align}
where $C$ is a constant independent of $v$.
\end{theorem}

\begin{proof}
Replacing $v(\cdot)$ by $v(\cdot,z)$ in the proof of Theorem \ref{thm: PJ w11 bound 2d}, one can show from Equations \eqref{eq: PJ-w11 eq4 2d} and \eqref{eq: PJ-w11 eq5 2d} that for almost every $z\in[-1,1]$:
\begin{align}
\|v_h(\cdot,z)\|_{L^2(h,1)} &\le 2\left(3-\sqrt{2}\right)\pi^{-1/2}h^{-1/2} \|g(\cdot,z)\|_{L^2(0,1)}, \label{eq: PJ-w11 eq1 3d} \\
\|v(\cdot,z)\|_{L^2(0,h)} &\le h^{1/2}\|D_1v(\cdot,z)\|_{L^1(0,1)}. \label{eq: PJ-w11 eq2 3d}
\end{align}
The consequence of Equation \eqref{eq: PJ-w11 eq1 3d} is immediate:
\begin{align}
\|v_h\|_{L^2(\Omega_h)} &\le 2\left(3-\sqrt{2}\right)\pi^{-1/2}h^{-1/2} \|g\|_{L^2(\Omega)}. \label{eq: PJ-w11 eq3 3d}
\end{align}
To obtain an analogy of Equation \eqref{eq: PJ-w11 eq5 2d} from Equation \eqref{eq: PJ-w11 eq2 3d}, we apply the $L^p$ interpolation theorem and embedding theorem, as well as Poincar\'{e}'s inequality in 2D:
\begin{align}
\|v\|_{L^2(\Omega\backslash\Omega_h)} &\le \|v\|_{L^1(\Omega\backslash\Omega_h)}^{1/2}\|v\|_{L^\infty(\Omega\backslash\Omega_h)}^{1/2} \le (2h)^{1/4}\|v\|_{L^2(\Omega\backslash\Omega_h)}^{1/2}\|v\|_{L^\infty(\Omega\backslash\Omega_h)}^{1/2} \nonumber \\
&\le (2h)^{1/4}\|Dv\|_{L^1(\Omega\backslash\Omega_h)}^{1/2}\|v\|_{L^\infty(\Omega\backslash\Omega_h)}^{1/2} \le (2h)^{1/4}\|Dv\|_{L^1(\Omega)}^{1/2}\|v\|_{L^\infty(\Omega)}^{1/2}, \label{eq: PJ-w11 eq4 3d}
\end{align}
where the $2h$ factor comes from the measure of the set $\Omega\backslash\Omega_h$.

By the triangle inequality and Equations \eqref{eq: basic bound 3d} and \eqref{eq: PJ-w11 eq3 3d}-\eqref{eq: PJ-w11 eq4 3d}, we have
\begin{align}
\|v\|_{L^2(\Omega)} &\le \|v\|_{L^2(\Omega\backslash\Omega_h)} + \|v-v_h\|_{L^2(\Omega_h)} + \|v_h\|_{L^2(\Omega_h)} \nonumber \\
&\le \left((4/3)^{1/4}+2^{1/4}\right)h^{1/4} \|v\|_{L^\infty(\Omega)}^{1/2}\|Dv\|_{L^1(\Omega)}^{1/2} + 2\left(3\sqrt{2}-2\right)\pi^{-1/2}h^{-1/2}\|g\|_{L^2(\Omega)} \nonumber \\
&\le 4\left((4/3)^{1/4}+2^{1/4}\right)h^{1/4} \|v\|_{L^\infty(\Omega)}^{1/2}\|Dv\|_{L^1(\Omega)}^{1/2} + 2\left(3\sqrt{2}-2\right)\pi^{-1/2}h^{-1/2}\|g\|_{L^2(\Omega)}, \label{eq: PJ-w11 eq5 3d}
\end{align}
where we have slightly extended the bound in the last step so that Equation \eqref{eq: PJ-w11 eq5 3d} can be optimized over $h\in(0, 1/2]$. The value $h^*$ that minimizes the right-hand side of Equation \eqref{eq: PJ-w11 eq5 3d} is given by
\begin{align}
h^* &= \left(\dfrac{\pi^{-1/2}\left(3\sqrt{2}-2\right)\|g\|_{L^2(\Omega)}}{\left((4/3)^{1/4}+2^{1/4}\right) \|v\|_{L^\infty(\Omega)}^{1/2}\|Dv\|_{L^1(\Omega)}^{1/2}}\right)^{4/3}. \label{eq: PJ-w11 eq6 3d}
\end{align}
We now verify that $h^*\in(0, 1/2]$, which depends on the factor $\|g\|_{L^2(\Omega)}\|v\|_{L^\infty(\Omega)}^{-1/2}\|Dv\|_{L^1(\Omega)}^{-1/2}$. Using the same derivation as in Equation \eqref{eq: PJ-w11 eq8 2d}, we have, for $(x,z)\in\Omega$,
\begin{align*}
g(x,z) = - \dfrac{2}{\sqrt{\pi}} \int_x^1 D_1v(y,z)\sqrt{|x-y|}\dd{y}.
\end{align*}
We first bound $\|g\|_{L^\infty(\Omega)}$ by $\|Dv\|_{L^\infty(\Omega)}$. Equation \eqref{eq: def abelJ 3d} implies that for $(x,z)\in\Omega$,
\begin{align*}
|g(x,z)| \le \dfrac{1}{\sqrt{\pi}} \int_x^1 \dfrac{|v(r,z)|}{\sqrt{r-x}}\dd{r} \le \left(\dfrac{1}{\sqrt{\pi}} \int_x^1 \dfrac{1}{\sqrt{r-x}}\dd{r}\right)\|v\|_{L^\infty(\Omega)} = \dfrac{2\sqrt{1-x}}{\sqrt{\pi}}\|v\|_{L^\infty(\Omega)},
\end{align*}
and thus:
\begin{align*}
\|g\|_{L^\infty(\Omega)} \le \dfrac{2}{\sqrt{\pi}}\|v\|_{L^\infty(\Omega)}.
\end{align*}
We then bound $\|g\|_{L^1(\Omega)}$ by $\|Dv\|_{L^1(\Omega)}$ by applying Young's inequality for convolutions with:
\begin{align*}
K(x,z) = - \dfrac{2}{\sqrt{\pi}}\sqrt{|x|}\mathds{1}_{[0,1]}(-x), \quad f(x,z) = D_1v(x,z)\mathds{1}_{[0,1]}(x).
\end{align*}
Thus, $g(x,z) = \int K(x-y,z)f(y,z)\dd{y}$ for $(x,z)\in\Omega$, and with $K$ and $f$ extended to $\R\times[-1,1]$:
\begin{align*}
\|g(\cdot,z)\|_{L^1(0,1)} = \|K*f(\cdot,z)\|_{L^1(0,1)} \le \|K(\cdot,z)\|_{L^1(\R)}\|f(\cdot,z)\|_{L^1(\R)} = \dfrac{4}{3\sqrt{\pi}} \|D_1v(\cdot,z)\|_{L^1(0,1)}.
\end{align*}
Therefore,
\begin{align*}
\|g\|_{L^1(\Omega)} \le \dfrac{4}{3\sqrt{\pi}} \|D_1v\|_{L^1(\Omega)} \le \dfrac{4}{3\sqrt{\pi}} \|Dv\|_{L^1(\Omega)}.
\end{align*}
Applying the $L^p$ interpolation theorem, we obtain:
\begin{align}
\|g\|_{L^2(\Omega)}^2 &\le \|g\|_{L^\infty(\Omega)}\|g\|_{L^1(\Omega)} \le \dfrac{8}{3\pi} \|v\|_{L^\infty(\Omega)}\|Dv\|_{L^1(\Omega)}.  \label{eq: PJ-w11 eq7 3d}
\end{align}
Combining Equations \eqref{eq: PJ-w11 eq6 3d} and \eqref{eq: PJ-w11 eq7 3d} yields:
\begin{align*}
h^{*} &=  \left(\dfrac{\pi^{-1/2}\left(3\sqrt{2}-2\right)\|g\|_{L^2(\Omega)}}{\left((4/3)^{1/4}+2^{1/4}\right) \|v\|_{L^\infty(\Omega)}^{1/2}\|Dv\|_{L^1(\Omega)}^{1/2}}\right)^{4/3} \le \left(\dfrac{12-4\sqrt{2}}{\sqrt{3}\left((4/3)^{1/4}+2^{1/4}\right)\pi} \right)^{4/3} \le \dfrac{1}{2}.
\end{align*}
Therefore, optimizing Equation \eqref{eq: PJ-w11 eq5 3d} over $h$ yields:
\begin{align}
\|v\|_{L^2(\Omega)} &\le \left(\dfrac{\pi^{-1/2}\left(3\sqrt{2}-2\right)}{\left((4/3)^{1/4}+2^{1/4}\right) }\right)^{4/3}\|v\|_{L^\infty(\Omega)}^{1/3}\|Dv\|_{L^1(\Omega)}^{1/3}\|g\|_{L^2(\Omega)}^{1/3}, \label{eq: PJ-w11 eq8 3d}
\end{align}
which gives Equation \eqref{eq: PJ w11 bound 3d}. 
\end{proof}


We now extend the result in Theorem \ref{thm: PJ w11 bound 3d} to obtain an $L^2$-stability estimate for $BV$ solutions to Problem \eqref{eq: g=Jv model} via a density argument with Lipschitz continuous functions.

\begin{theorem} \label{thm: PJ bv bound 3d}
If $v\in BV(\Omega)\cap L^\infty(\Omega)$ with $\supp(v)\subset[0,1)\times[-1,1]$ and $\abelJ v = g$, we have:
\begin{align}
\|v\|_{L^2(\Omega)} &\le C\|v\|_{L^\infty(\Omega)}^{1/3}\|v\|_{TV(\Omega)}^{1/3}\|g\|_{L^2(\Omega)}^{1/3}, \label{eq: PJ bv bound 3d}
\end{align}
where $C$ is a constant independent of $v$.
\end{theorem}

\begin{proof}
Note that the assumption $v\in L^\infty(\Omega)$ implies that $v\in L^2(\Omega)$. By Theorem \ref{thm: approximating sequence}, there exists a sequence of functions $\{v_k\}_{k=1}^\infty\subset \lip(\Omega)$ such that 
\begin{subequations}
\begin{align}
&\|v_k-v\|_{L^2(\Omega)}\to0 \quad {\rm as}\,\,k\to\infty, \label{eq: PJ-bv eq1a 3d} \\
&\|v_k\|_{TV(\Omega)}\to\|v\|_{TV(\Omega)} \quad {\rm as}\,\,k\to\infty, \label{eq: PJ-bv eq1b 3d} \\
{\rm and}\quad &\|v_k\|_{L^\infty(\Omega)} \le \|v\|_{L^\infty(\Omega)}\quad{\rm for\,\,all}\,\,k. \label{eq: PJ-bv eq1c 3d}
\end{align}\label{eq: PJ-bv eq1 3d}
\end{subequations}
Define $g_k:=\abelJ v_k$. By Theorem \ref{thm: PJ w11 bound 3d}, 
\begin{align}
\|v_k\|_{L^2(\Omega)} &\le C\|v_k\|_{L^\infty(\Omega)}^{1/3}\|Dv_k\|_{L^1(\Omega)}^{1/3}\|g_k\|_{L^2(\Omega)}^{1/3}, \label{eq: PJ-bv eq2 3d}
\end{align}
where $C$ is a constant independent of the choice of the approximating sequence. For each $k\geq 1$, since $v_k\in\lip(\Omega)$, $Dv_k$ exists almost everywhere, and thus condition \eqref{eq: PJ-bv eq1b 3d} implies:
\begin{align}
\|Dv_k\|_{L^1(\Omega)}\to\|v\|_{TV(\Omega)} \quad {\rm as}\,\,k\to\infty. \label{eq: PJ-bv eq3 3d}
\end{align}
On the other hand, applying Corollary \ref{cor: J lp continuity 3d} with $p=2$ and condition \eqref{eq: PJ-bv eq1a 3d}, we have:
\begin{align}
\|g_k-g\|_{L^2(\Omega)} = \|\abelJ(v_k-v)\|_{L^2(\Omega)} \le \dfrac{2}{\sqrt{\pi}}\|v_k-v\|_{L^2(\Omega)} \to0 \label{eq: PJ-bv eq4 3d}
\end{align}
as $k\to\infty$. Therefore, by Equations \eqref{eq: PJ-bv eq1 3d}-\eqref{eq: PJ-bv eq4 3d},
\begin{align*}
\|v\|_{L^2(\Omega)} &= \lim_{k\to\infty}\|v_k\|_{L^2(\Omega)} \\
&= \lim_{k\to\infty}C\|v_k\|_{L^\infty(\Omega)}^{1/3}\|Dv_k\|_{L^1(\Omega)}^{1/3}\|g_k\|_{L^2(\Omega)}^{1/3} \\
&\le C\|v\|_{L^\infty(\Omega)}^{1/3}\|v\|_{TV(\Omega)}^{1/3}\|g\|_{L^2(\Omega)}^{1/3},
\end{align*}
which gives Equation \eqref{eq: PJ bv bound 3d}.
\end{proof}


Using the same density argument, one can prove the following theorem from Equations \eqref{eq: PJ bv bound 3d} and \eqref{eq: r-bound to xy-bound 3d}, thus extending the $L^2$-stability estimate to Problem \eqref{eq: f=Au model}.


\begin{theorem} \label{thm: PA bv bound 3d}
Let $u:U\subset\R^3\to\R$ be an axisymmetric function such that, as a function of $(r,z)$, $u\in BV(\Omega)\cap L^\infty(\Omega)$ and $\supp(u)\subset[0,1)\times[-1,1]$. If $\abel u = f$, we have:
\begin{align}
\|u\|_{L^2(U)} &\le C\|u\|_{L^\infty(\Omega)}^{1/3}\|u\|_{TV(\Omega)}^{1/3}\|f\|_{L^2(\Omega)}^{1/3}, \label{eq: PA bv bound 3d}
\end{align}
where $C$ is a constant independent of $u$,
\begin{align*}
\|u\|_{L^2(U)} &:= \left(\int_{-1}^1\iint_{B(0,1)} |u(x,y,z)|^2\dd{x}\dd{y}\dd{z}\right)^{1/2}, 
\end{align*}
and $\|u\|_{TV(\Omega)}$ is defined by Equation \eqref{eq: def TV norm 3d}.
\end{theorem}


\subsection{Error Estimate for a TV Regularized Model}
\label{sec: error}

Theorems \ref{thm: PA bv bound 2d} and \ref{thm: PA bv bound 3d} provide an $L^2$-stability estimate for $BV$ solutions to Problem \eqref{eq: f=Au model} with control given by the $L^2$ norm of the data. Let $f_0$ be the unknown noise-free data (we can assume that $f_0\in BV$), and $f$ be the noisy data, where $f = f_0+\eta$, $\eta\sim\normal(0,\sigma^2)$. Let $u_0$ be a bounded $BV$ solution to Problem \eqref{eq: f=Au model}, {\it i.e.} $f_0=\abel u_0$, and $u^*$ be the unique bounded solution of Problem \eqref{eq: TV model}; the uniqueness is guaranteed by Theorem 3.1 in \cite{Acar94}. Define the Abel transform of $u^*$ as $f^*:=\abel u^*$. We have that $f^*\in L^2$ by, for example, Equation \eqref{eq: PJ-w11 eq7 3d}. Define the sets $\mathcal{S}_1$ and $\mathcal{S}_2$ as follows:
\begin{align*}
\mathcal{S}_1(c) &:= \{ u\in BV(0,1): \|u\|_{TV(0,1)}\leq c \text{ and }  \|\abel u\|_{L^2(0,1)} < \infty  \}, \\
\mathcal{S}_2(c,M) &:= \{ u\in BV(\Omega): \|u\|_{TV(\Omega)}\leq c, \ \ \|u\|_{L^\infty(\Omega)}\leq M, \text{ and } \|\abel u\|_{L^2(\Omega)} < \infty \}.
\end{align*}
Then the following corollaries are consequences of Theorems \ref{thm: PA bv bound 2d} and \ref{thm: PA bv bound 3d}, respectively, which provide an error estimate over the sets above. 


\begin{corollary} \label{cor: Ptv error 2d}
Assume that the data $f_0$ and $f$ are defined on $[0,1]\subset\R$, the conditions for $u^*$ and $u_0$ are as previously stated, and $u, u_0 \in \mathcal{S}_1(c)$ for some constant $c$. Then 
\begin{align*}
\|u^*-u_0\|_{L^2(B(0,1))} \le C \left(\|f^*-f\|_{L^2(0,1)} + \sigma\right)^{1/2},
\end{align*}
where $C$ is a constant depending on $\sqrt{c}$.
\end{corollary}

\begin{proof}
By assumption:
\begin{align*}
\|u^*\|_{TV(0,1)},\|u_0\|_{TV(0,1)}\le c.
\end{align*} 
Then by Theorem \ref{thm: PA bv bound 2d},
\begin{align*}
\|u^*-u_0\|_{L^2(B(0,1))} &\le C\|u^*-u_0\|_{TV(0,1)}^{1/2}\|\abel(u^*-u_0)\|_{L^2(0,1)}^{1/2} \\
&\le C\sqrt{2c} \|f^*-f_0\|_{L^2(0,1)}^{1/2} \\
&\le C\sqrt{2c} \left(\|f^*-f\|_{L^2(0,1)} + \|f-f_0\|_{L^2(0,1)} \right)^{1/2} \\
&= C\sqrt{2c} \left(\|f^*-f\|_{L^2(0,1)} + \sigma \right)^{1/2}.
\end{align*}
\end{proof}


\begin{corollary} \label{cor: Ptv error 3d}
Assume that the data $f_0$ and $f$ are defined on $\Omega\subset\R^2$, the conditions for $u^*$ and $u_0$ are as previously stated, and $u, u_0 \in \mathcal{S}_2(c,M)$ for some constants $c$ and $M$. Then 
\begin{align*}
\|u^*-u_0\|_{L^2(U)} \le C \left(\|f^*-f\|_{L^2(\Omega)} + \sqrt{2}\sigma\right)^{1/3},
\end{align*}
where $C$ is a constant depending on $(cM)^{1/3}$.
\end{corollary}

\begin{proof}
Note that since the size of the domain $\Omega$ is equal to 2, we have $\|f-f_0\|_{L^2(\Omega)}^2 = 2\sigma^2$. By assumption:
\begin{align*}
\|u^*\|_{TV(\Omega)},\|u_0\|_{TV(\Omega)}\le c, \quad \text{and} \quad \|u^*\|_{L^\infty(\Omega)},\|u_0\|_{L^\infty(\Omega)}\le M.
\end{align*}
Then by Theorem \ref{thm: PA bv bound 3d},
\begin{align}
\|u^*-u_0\|_{L^2(U)} & \le C\|u^*-u_0\|_{L^\infty(\Omega)}^{1/3}\|u^*-u_0\|_{TV(\Omega)}^{1/3}\|\abel(u^*-u_0)\|_{L^2(\Omega)}^{1/3} \nonumber \\
&\le C(4cM)^{1/3} \|f^*-f_0\|_{L^2(\Omega)}^{1/3} \nonumber \\
&\le C(4cM)^{1/3} \left(\|f^*-f\|_{L^2(\Omega)} + \|f-f_0\|_{L^2(\Omega)} \right)^{1/3} \nonumber \\
&= C(4cM)^{1/3} \left(\|f^*-f\|_{L^2(\Omega)} + \sqrt{2}\sigma \right)^{1/3}. \label{eq: Ptv error eq1 3d}
\end{align}
\end{proof}


\begin{remark} \label{rem: Ptv convergence}
Theorem 5.1 in \cite{Acar94} provides a convergence result for the solutions to a sequence of perturbed linear inverse problems. In particular, for 2D axisymmetric solutions, assume that the data $f_0$ is defined on $[0,1]\subset\R$. Let $\{f_k\}_{k=1}^\infty$ be a sequence of perturbed data, where $f_k = f_0+\eta_k$, $\eta_k\sim\normal(0,\sigma_k^2)$. Let $\{u_k\}_{k=1}^\infty$ be the solutions obtained by minimizing:
\begin{align*}
\|u\|_{TV(0,1)} + \dfrac{\lambda_k}{2}\|\abel u-f_k\|_{L^2(0,1)}^2
\end{align*}
over $u\in BV(0,1)$. Suppose $\|f_k-f\|_{L^2(0,1)}\to0$, and $\lambda_k\to\infty$ at a rate such that $\lambda_k\|\abel u_0-f_k\|_{L^2(0,1)}^2$ remains bounded. Then $u_k\to u_0$ strongly in $L^2$.
And for 3D axisymmetric solutions, assume that the data $f_0$ is defined on $\Omega\subset\R^2$. Let $\{f_k\}_{k=1}^\infty$ be defined as before. Let $\{u_k\}_{k=1}^\infty$ be the solutions obtained by minimizing:
\begin{align*}
\|u\|_{TV(\Omega)} + \dfrac{\lambda_k}{2}\|\abel u-f_k\|_{L^2(\Omega)}^2
\end{align*}
over $u\in BV(\Omega)$. Suppose $\|f_k-f\|_{L^2(\Omega)}\to0$, and $\lambda_k\to\infty$ at a rate such that $\lambda_k\|\abel u_0-f_k\|_{L^2(\Omega)}^2$ remains bounded. Then $u_k\rightharpoonup u_0$ weakly in $L^2$.
\end{remark}


\section{Examples}
\label{sec: example}

In this section, two numerical examples are detailed and used to verify the theory from Section \ref{sec: theory}. In each case, an approximation is obtain be solving  Problem \eqref{eq: TV model} in the presence of additive Gaussian noise and the error bounds are verified numerically.

We consider two synthetic axisymmetric density functions which are compactly supported in the cylindrical domain $U$. Let $U_h$ and $V_h$ be discretization of $[-1,1]\times[-1,1]\times[0,1]\subset\R^3$ and $\Omega$, respectively, with grid-spacing $h$ equal to $1/128$. To solve Problem \eqref{eq: TV model} numerically, consider the following discrete minimization problem:
\begin{align}
\min_{u\in V_h}\quad & \|\nabla_h u\|_{\ell^1(V_h\times V_h)} + \dfrac{\lambda}{2}\|Au-f\|_{\ell^2(V_h)}^2, \label{eq: discrete TV model}\tag{\mbox{P$_{{\rm TV},h}$}}
\end{align}
which can be solved via the primal-dual algorithm \cite{Chambolle11}. Further details about the discretization and the numerical method can be found in Appendix \ref{sec: method}.

In both cases, we consider piecewise constant densities $u_0$. Figure \ref{fig: 1a} shows the level sets of the density along with a planar slice. Each of the level sets have a rough-boundary; however, the function is still in $BV$. The ``observed'' data $f$ is given in Figure \ref{fig: 1b}, where $f=f_0+\eta$, $f_0=\abel u_0$, $\eta\sim\normal(0,\sigma^2)$, and $\sigma^2=0.05\%\times\|f_0\|_{L^\infty(U)}$. Figure \ref{fig: 1c} displays the approximate solution $u^*$ which is the discrete minimizer of Problem \eqref{eq: discrete TV model} given measured data $f$ as shown in Figure \ref{fig: 1b}. It can be seen that the boundaries between constant density regions are well-recovered, except near the origin. This is due to high-variations near the origin which are penalized (strongly) by the TV semi-norm.  In Figure \ref{fig: 1d}, we display the approximate solution $u^*$ corresponding to a lower noise level, \textit{i.e.} $\sigma^2 = 0.01\%\times\|f_0\|_{L^\infty(U)}$. As the noise decreases, the level sets become better-resolved.

%

\begin{figure}[b!]
  \centering
  \subfloat[$u_0=u_0(x,y,z)$\label{fig: 1a}]{\includegraphics[scale=0.18]{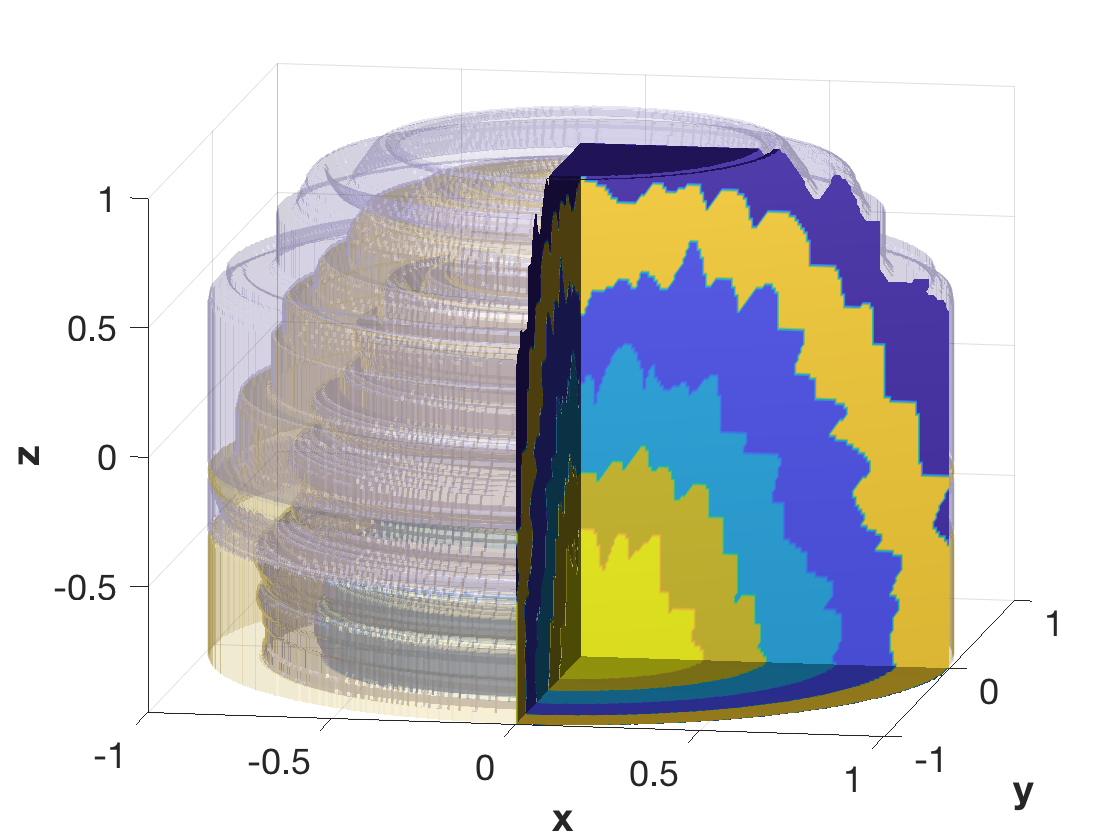}}
  \subfloat[$f$ with $\sigma^2=0.05\%\times\|f_0\|_{L^\infty(U)}$\label{fig: 1b}]{\includegraphics[scale=0.18]{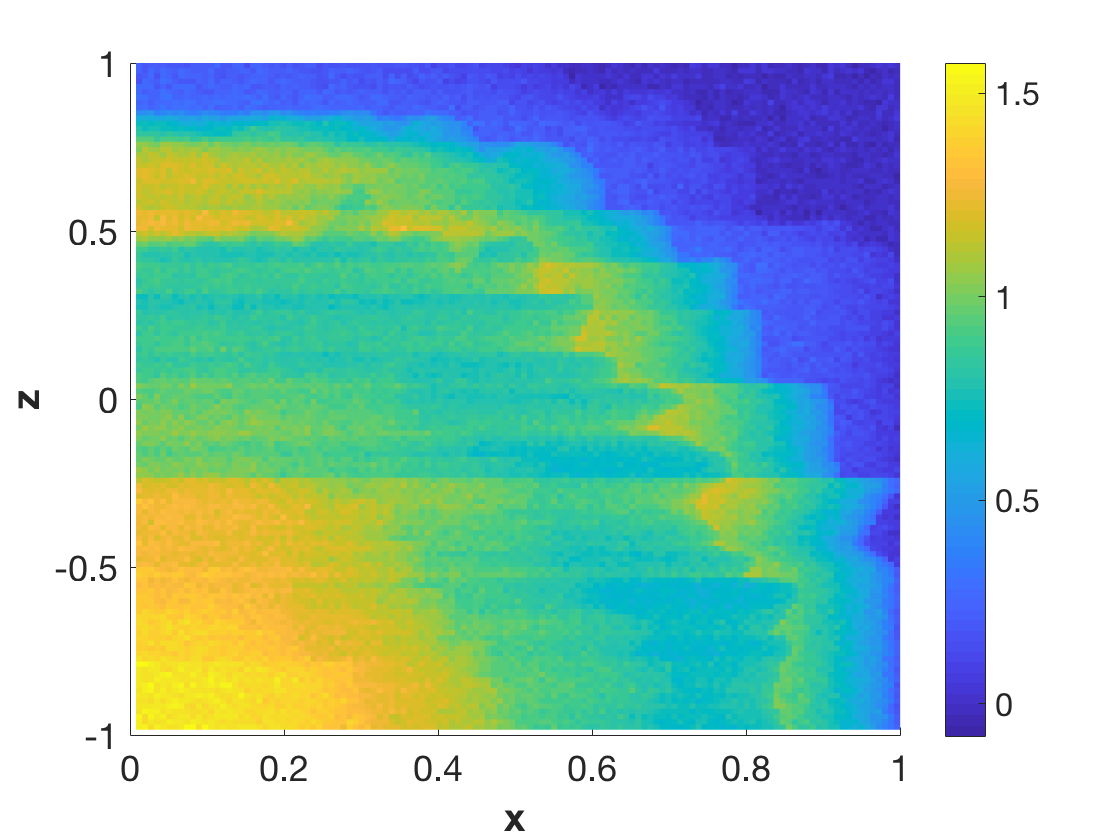}}
  
  \subfloat[$u^*=u^*(x,y,z)$, $\sigma^2=0.05\%\times\|f_0\|_{L^\infty(U)}$\label{fig: 1c}]{\includegraphics[scale=0.18]{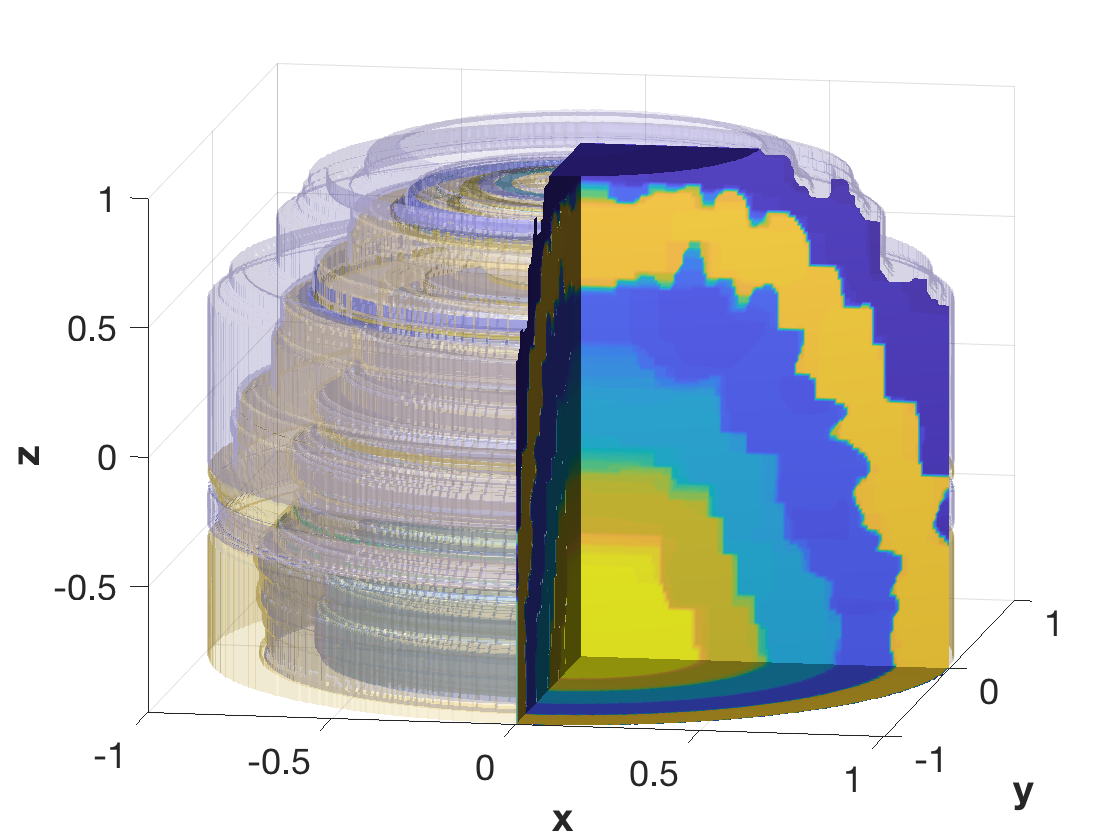}}
  \subfloat[$u^*=u^*(x,y,z)$, $\sigma^2=0.01\%\times\|f_0\|_{L^\infty(U)}$\label{fig: 1d}]{\includegraphics[scale=0.18]{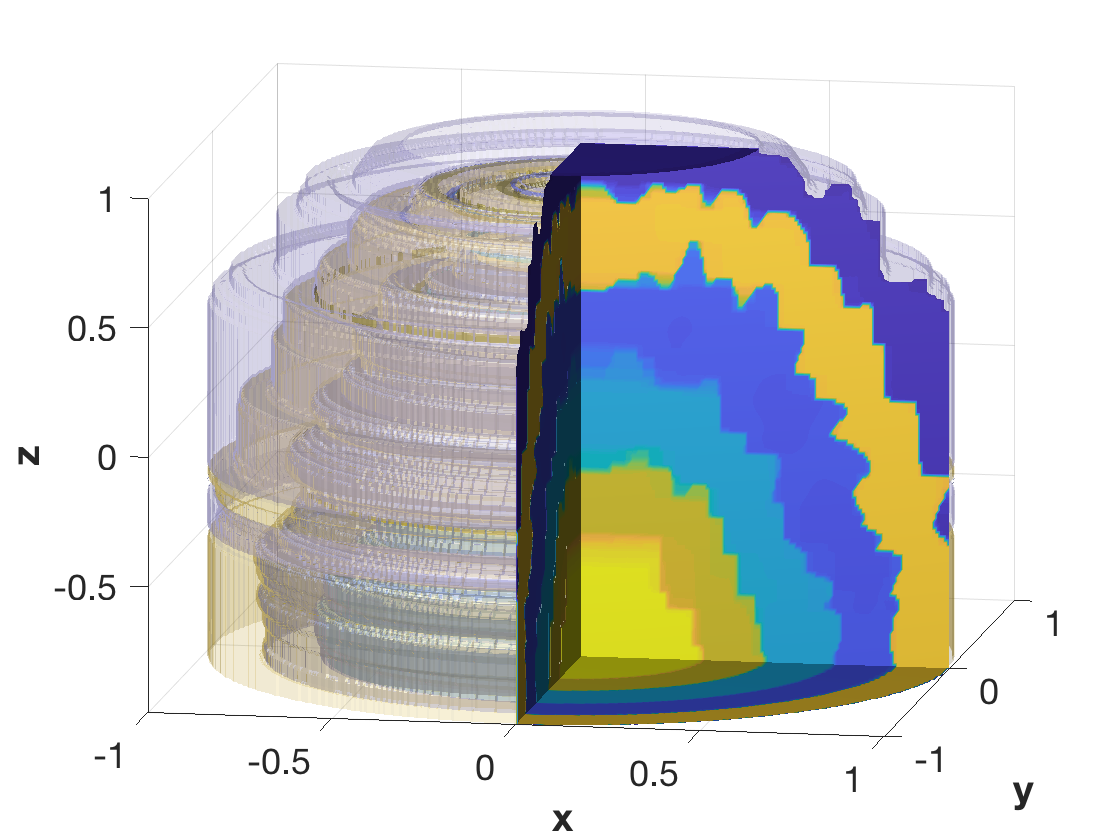}}
  \caption{Example 1: (a)  Level sets and planar slice of the original density $u_0$,  (b) the noisy observation $f$, where $f=\abel u_0+\eta$, $\eta\sim\normal(0,\sigma^2)$, (c-d) recovered data using Problem \eqref{eq: discrete TV model} when the variance of the noise $\sigma^2$ is $0.05\%\times\|f_0\|_{L^\infty(U)}$ and $0.01\%\times\|f_0\|_{L^\infty(U)}$ respectively. }
  \label{fig: ex1}
\end{figure}

For the second example, we consider a piecewise constant density with four disjoint topological components. Figure \ref{fig: 2a} shows the level sets of the original density $u_0$. The noisy ``observed'' data $f$ is given in Figure \ref{fig: 2b}, where $f=f_0+\eta$, $f_0=\abel u_0$, $\eta\sim\normal(0,\sigma^2)$, and $\sigma^2=0.05\%\times\|f_0\|_{L^\infty(U)}$. Figure \ref{fig: 2c} and \ref{fig: 2d}  display the approximate solution $u^*$ which is the discrete minimizer of Problem \eqref{eq: discrete TV model} given noise level $\sigma^2=0.05\%\times\|f_0\|_{L^\infty(U)}$ and $\sigma^2=0.002\%\times\|f_0\|_{L^\infty(U)}$, respectively. As the noise level decreases, the high-curvature regions (the lower tip of the yellow and blue components) are better-resolved.

%

\begin{figure}[b!]
  \centering
  \subfloat[$u_0=u_0(x,y,z)$\label{fig: 2a}]{\includegraphics[scale=0.18]{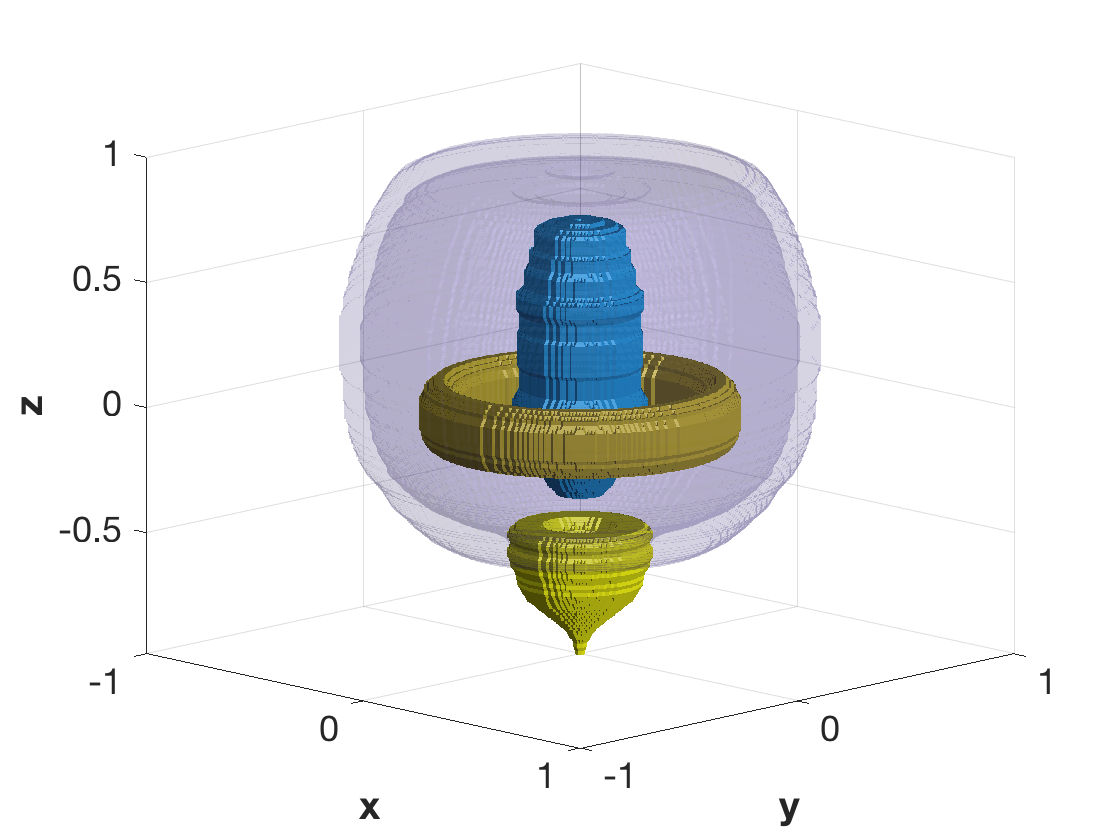}}
  \subfloat[$f$ with $\sigma^2= 0.05\%\times\|f_0\|_{L^\infty(U)}$\label{fig: 2b}]{\includegraphics[scale=0.18]{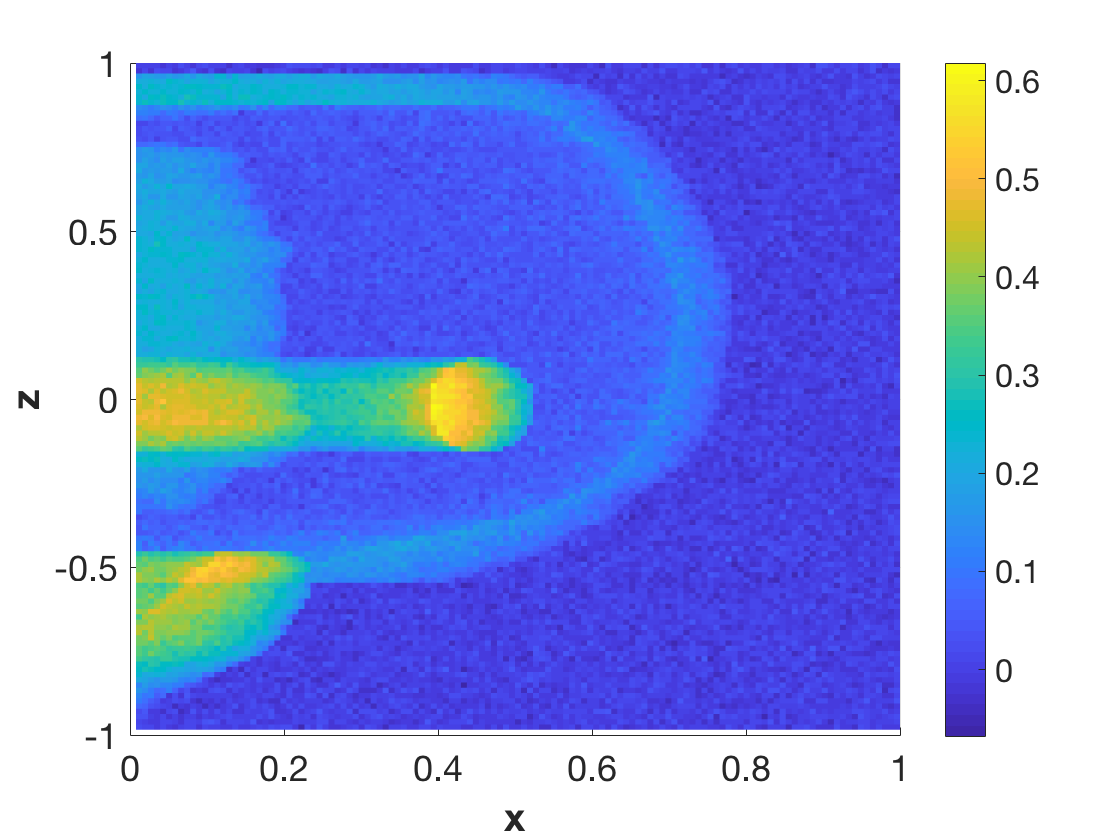}}
  
  \subfloat[$u^*=u^*(x,y,z)$, $\sigma^2=0.05\%\times\|f_0\|_{L^\infty(U)}$\label{fig: 2c}]{\includegraphics[scale=0.18]{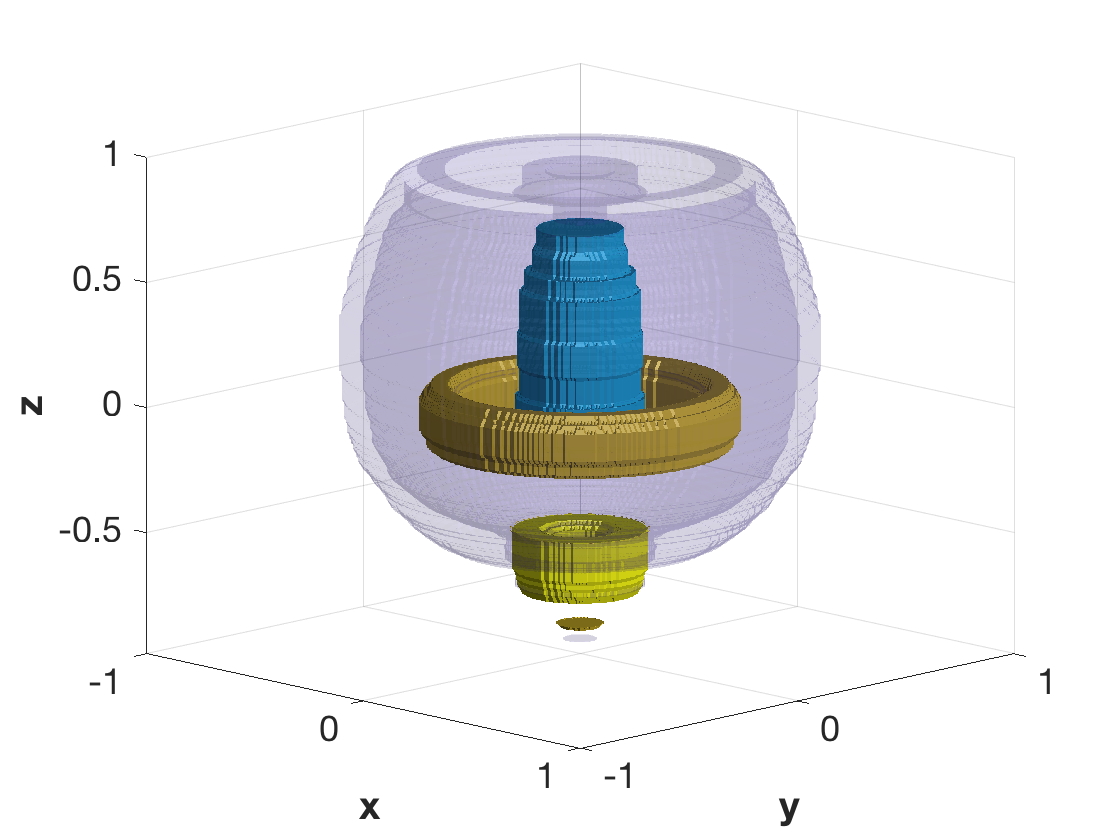}} 
  \subfloat[$u^*=u^*(x,y,z)$, $\sigma^2=0.002\%\times\|f_0\|_{L^\infty(U)}$\label{fig: 2d}]{\includegraphics[scale=0.18]{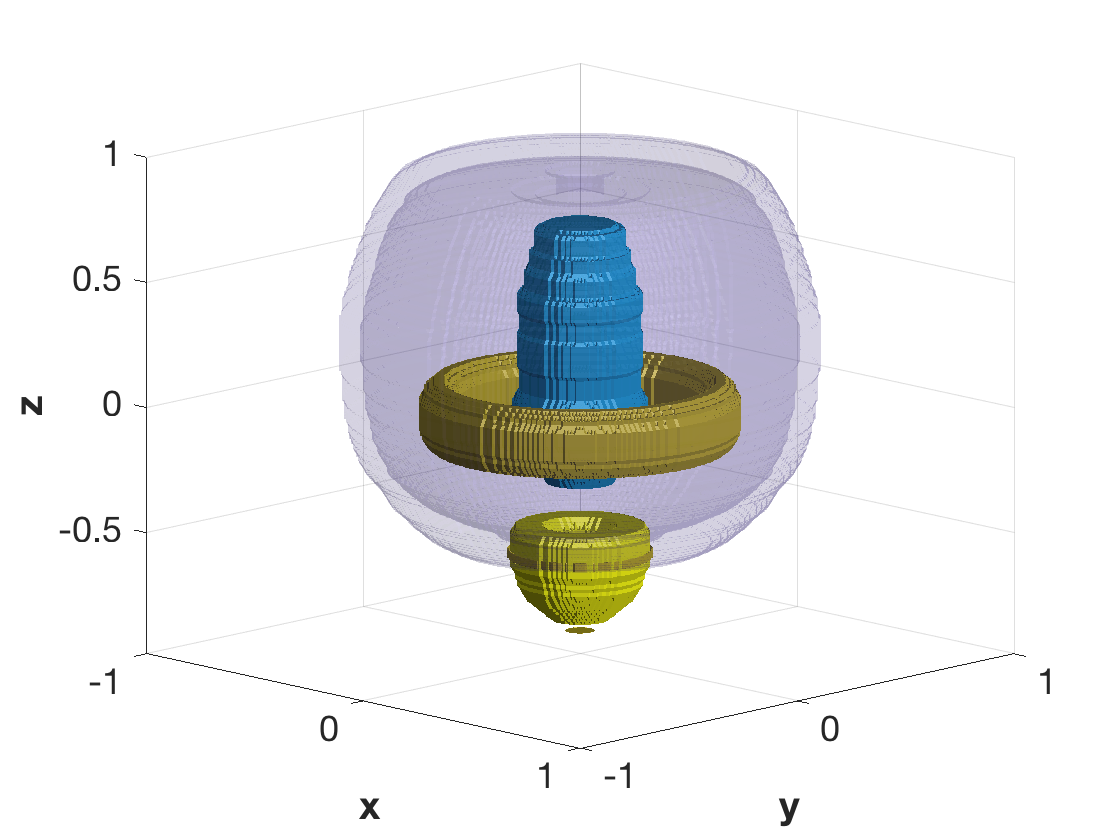}}
  \caption{Example 2: (a)  Level sets and planar slice of the original density $u_0$,  (b) the noisy observation $f$, where $f=\abel u_0+\eta$, $\eta\sim\normal(0,\sigma^2)$, (c-d) recovered data using Problem \eqref{eq: discrete TV model} when the variance of the noise $\sigma^2$ is $0.05\%\times\|f_0\|_{L^\infty(U)}$ and $0.002\%\times\|f_0\|_{L^\infty(U)}$ respectively. }
  \label{fig: ex2}
\end{figure}

For each of the examples, we solve Problem \eqref{eq: discrete TV model} with difference $\sigma$ values. The parameters used in the computational experiments are listed in Table \ref{tab: parameter}.

\begin{table}[b!]
  \caption{Parameters corresponding to Examples 1 and 2, Figures~\ref{fig: ex1} and ~\ref{fig: ex2}, respectively.}
  \label{tab: parameter}
  \centering
  \subfloat[Example 1]{
  \begin{tabular}{|c|c|c|c|c|} \hline
  Parameters of the data & \multicolumn{4}{c|}{Parameters of the algorithm} \\ \hline
   $\sigma^2$ & $\lambda$ & $\tau$ & $\gamma$ & Total iterations \\ \hline
   $0.25\%\times\|f_0\|_{L^\infty(U)}$ & 50 & 0.2 & 0.2 & 5000 \\ \hline
   $0.05\%\times\|f_0\|_{L^\infty(U)}$ & 80 & 0.2 & 0.2 & 5000 \\ \hline
   $0.01\%\times\|f_0\|_{L^\infty(U)}$ & 120 & 0.2 & 0.2 & 5000 \\ \hline
   $0.002\%\times\|f_0\|_{L^\infty(U)}$ & 170 & 0.2 & 0.2 & 5000 \\ \hline
  \end{tabular}}\\
  \subfloat[Example 2]{
  \begin{tabular}{|c|c|c|c|c|} \hline
  Parameters of the data & \multicolumn{4}{c|}{Parameters of the algorithm} \\ \hline
   $\sigma^2$ & $\lambda$ & $\tau$ & $\gamma$ & Total iterations \\ \hline
   $0.25\%\times\|f_0\|_{L^\infty(U)}$ & 60 & 0.4 & 0.4 & 5000 \\ \hline
   $0.05\%\times\|f_0\|_{L^\infty(U)}$ & 90 & 0.4 & 0.4 & 5000 \\ \hline
   $0.01\%\times\|f_0\|_{L^\infty(U)}$ & 150 & 0.2 & 0.2 & 5000 \\ \hline
   $0.002\%\times\|f_0\|_{L^\infty(U)}$ & 180 & 0.2 & 0.2 & 5000 \\ \hline
  \end{tabular}}
\end{table}

To verify the error bound from Section \ref{sec: theory}, define the following discrete quantities:
\begin{align*}
c &:= \max\left\{\|\nabla_h u^*\|_{\ell^1(V_h\times V_h)},\|\nabla_h u_0\|_{\ell^1(V_h\times V_h)}\right\}, \\
M &:= \max\left\{\|u^*\|_{\ell^\infty(V_h)},\|u_0\|_{\ell^\infty(V_h)}\right\}, \\
M_1 &:= \left(\|f^*-f\|_{\ell^2(V_h)} + \|f-f_0\|_{\ell^2(V_h)} \right)^{1/3}, \\
C^* &:= \dfrac{\|u^*-u_0\|_{\ell^2(U_h)}}{M_1\, (4cM)^{1/3}}.
\end{align*}
Note that, in practice, an upper bound of $\|f-f_0\|_{\ell^2(V_h)}$ could be estimated from the data without knowledge of $f_0$. The values used for error estimate of each experiment are listed in Table \ref{tab: result}. From Equations \eqref{eq: PJ-w11 eq8 3d}, \eqref{eq: PA bv bound 3d}, \eqref{eq: Ptv error eq1 3d}, and \eqref{eq: r-bound to xy-bound 3d}, it is expected that $C^*\leq 1.07$. This is in fact the case numerically, thereby providing additional support for Corollary \ref{cor: Ptv error 3d}. Moreover, from Tables \ref{tab: parameter} and \ref{tab: result}, it can be seen that by choosing the parameter $\lambda$, the quantity $\|u^*-u_0\|_{\ell^2(U_h)}$ can be made decreasing as $\sigma$ decreases. This provides numerical support for Remark \ref{rem: Ptv convergence}. Lastly, it is worth noting that the numerical experiments suggest better control of the error than what was shown theoretically. 

%
%

\begin{table}[b!]
  \caption{Discrete quantities used to verify the error bound on Examples 1 and 2, Figures~\ref{fig: ex1} and ~\ref{fig: ex2}, respectively.}
  \label{tab: result}
  \centering
  
  \subfloat[Example 1]{
  \begin{tabular}{|c|c|c|c|c|c|c|} \hline
  $\sigma^2$ & $\|u^*-u_0\|_{\ell^2(U_h)}$ & $\|f^*-f\|_{\ell^2(V_h)}$ & $M_1$ & $c$ & $M$ & $C^*$ \\ \hline
  $0.25\%\times\|f_0\|_{L^\infty(U)}$ & 0.0914 & 0.0606 & 0.4961 & 0.0285 & 1 & 0.3796 \\ \hline
  $0.05\%\times\|f_0\|_{L^\infty(U)}$ & 0.0596 & 0.0271 & 0.3782 & 0.0285 & 1 & 0.3249 \\ \hline
  $0.01\%\times\|f_0\|_{L^\infty(U)}$ & 0.0378 & 0.0126 & 0.2917 & 0.0285 & 1 & 0.2669 \\ \hline
  $0.002\%\times\|f_0\|_{L^\infty(U)}$ & 0.0317 & 0.0064 & 0.2278 & 0.0285 & 1 & 0.2867 \\ \hline
  \end{tabular}
  }
  
  \subfloat[Example 2]{
  \begin{tabular}{|c|c|c|c|c|c|c|} \hline
  $\sigma^2$ & $\|u^*-u_0\|_{\ell^2(U_h)}$ & $\|f^*-f\|_{L^2(\Omega)}$ & $M_1$ & $c$ & $M$ & $C^*$ \\ \hline
  $0.25\%\times\|f_0\|_{L^\infty(U)}$ & 0.0518 & 0.0389 & 0.4268 & 0.0194 & 1 & 0.2845 \\ \hline
  $0.05\%\times\|f_0\|_{L^\infty(U)}$ & 0.0363 & 0.0177 & 0.3269 & 0.0194 & 1 & 0.2601 \\ \hline
  $0.01\%\times\|f_0\|_{L^\infty(U)}$ & 0.0270 & 0.0086 & 0.2537 & 0.0194 & 1 & 0.2490 \\ \hline
  $0.002\%\times\|f_0\|_{L^\infty(U)}$ & 0.0201 & 0.0044 & 0.1989 & 0.0194 & 1 & 0.2368 \\ \hline
  \end{tabular}
}
\end{table}


\section{Discussion}
\label{sec: conclusion}

In this work, the problem of recovering a $BV$ function from its Abel projection is analyzed. The difficulty in this problem is related to the the ill-conditioning of
the Abel inverse problem \eqref{eq: f=Au model} and the influence of noise, which is handled through a TV regularized model \eqref{eq: TV model}. We provide $L^2$-stability estimates for $BV$ solutions to Problem \eqref{eq: f=Au model} and error bounds from minimizers of Problem \eqref{eq: TV model}.  Additionally, numerical examples in three dimensions verify the theoretical results. These results provide theoretical guarantees on the recovery of data from (noisy) line-of-sight projections. 

In the future, we would like to generalize the theoretical results and derive optimal bounds. The theory provided in Section \ref{sec: theory} could be modified to provide estimates for other integral equations related to line-of-sight projections.  The stability bounds found in Section \ref{sec: theory} are sub-linear, and based on numerical observations,  may not be optimal. We are interested in improving, for example, the $1/3$ exponent in Equation \eqref{eq: PJ w11 bound 3d}. In addition, it would be worth investigating approximations of Problem \eqref{eq: f=Au model} with other variational models with linear-growth conditions on the gradient. Recovery guarantees of variational methods over $BV$ functions should follow from the analysis presented in this work.


\appendix


\section{$L^1$-Stability Estimates for $BV$ Solutions}
\label{sec: L1 bound}

In this section, we provide $L^1$-stability estimates for $BV$ solutions to Problems \eqref{eq: f=Au model} and \eqref{eq: g=Jv model}.


\begin{lemma} \label{lem: basic bounds 2d l1}
Let $v\in W^{1,1}(0,1)$. Let $h\in(0, 1/2]$ and define $v_h$ by Equation \eqref{eq: running average 2d}. Then the following estimate holds:
\begin{align}
\|v-v_h\|_{L^1(h,1)} \le 2^{-1}h\|v'\|_{L^1(0,1)} \label{eq: basic bound 2d l1}
\end{align}
\end{lemma}

\begin{proof}
We have shown in the proof of Lemma \ref{lem: basic bounds 2d} that $v_h - v = K*g$ on $[h,1]$, where
\begin{align*}
K(x) := \left(\dfrac{x}{h} -1\right)\1_{[0,h]}(x), \quad g(x) := v'(x)\mathds{1}_{[0,1]}(x).
\end{align*}
Extending the functions to $\R$ and applying Young's inequality for convolutions, we obtain:
\begin{align*}
\|v-v_h\|_{L^1(h,1)} = \|K*g\|_{L^1(h,1)} \leq \|K*g\|_{L^1(\R)} \le \|K\|_{L^1(\R)}\|g\|_{L^1(\R)} = 2^{-1}h\|v'\|_{L^1(0,1)},
\end{align*}
where the last equality can be calculated directly. This shows Equation \eqref{eq: basic bound 2d l1}.
\end{proof}


\begin{theorem} \label{thm: PJ w11 bound 2d l1}
If $v\in W^{1,1}(0,1)$ with $\supp(v)\subset[0,1)$ and $\abelJ v = g$, we have:
\begin{align}
\|v\|_{L^1(0,1)} \le C\|v'\|_{L^1(0,1)}^{1/3}\|g\|_{L^1(0,1)}^{2/3}, \label{eq: PJ w11 bound 2d l1}
\end{align}
where $C$ is a constant independent of $v$.
\end{theorem}

\begin{proof}
We have shown in the proof of Theorem \ref{thm: PJ w11 bound 2d} that
\begin{align}
\sqrt{\pi}hv_h(x) &= \int_{x-h}^x\dfrac{g(y)}{\sqrt{y-(x-h)}} \dd{y} + \int_x^1g(y) \left[\dfrac{1}{\sqrt{y-(x-h)}} - \dfrac{1}{\sqrt{y-x}}\right] \dd{y} \nonumber \\
&=: F_1(x) + F_2(x). \label{eq: PJ-w11 eq2 2d l1}
\end{align}
It can be seen from Equation \eqref{eq: PJ-w11 eq2 2d l1} that $F_1 = K_1*g$ and $F_2=K_2*g$ on $[h,1]$, where
\begin{align*}
K_1(x):= \dfrac{\1_{[0,h]}(x)}{\sqrt{h-x}}, \quad K_2(x):= \dfrac{\1_{[0,h]}(-x)}{\sqrt{h-x}} - \dfrac{\1_{[0,h]}(-x)}{\sqrt{|x|}},
\end{align*}
and we have extended the functions to $\R$. By Young's inequality for convolutions, we have $L^1(h,1)$ control over each term in Equation \eqref{eq: PJ-w11 eq2 2d l1}:
\begin{subequations}\label{eq: PJ-w11 eq3 2d l1}
\begin{align}
\|F_1\|_{L^1(h,1)} &= \|K_1*g\|_{L^1(h,1)} \le \|K_1\|_{L^1(\R)}\|g\|_{L^1(\R)} = 2h^{1/2}\|v'\|_{L^1(0,1)}, \\
\|F_2\|_{L^1(h,1)} &= \|K_2*g\|_{L^1(h,1)} \le \|K_2\|_{L^1(\R)}\|g\|_{L^1(\R)} = 2\left(2-\sqrt{2}\right)h^{1/2} \|g\|_{L^1(0,1)},
\end{align}
\end{subequations}
where the last equalities can be calculated directly.  By combining Equations \eqref{eq: PJ-w11 eq2 2d l1} and \eqref{eq: PJ-w11 eq3 2d l1}, we obtain the following:
\begin{align}
\|v_h\|_{L^1(h,1)} \le 2\left(3-\sqrt{2}\right)\pi^{-1/2}h^{-1/2} \|g\|_{L^1(0,1)}. \label{eq: PJ-w11 eq4 2d l1}
\end{align}
By applying the $L^p$ interpolation theorem and Poincar\'{e}'s inequality in 1D, we obtain:
\begin{align}
\|v\|_{L^1(0,h)} &\le h\|v\|_{L^\infty(0,h)} \le h\|v\|_{L^\infty(0,1)} \le h\|v'\|_{L^1(0,1)}. \label{eq: PJ-w11 eq5 2d l1}
\end{align}
To obtain an estimate in the ${L^1(0,1)}$ norm, we apply the triangle inequality and the results from Equations \eqref{eq: basic bound 2d l1} and \eqref{eq: PJ-w11 eq4 2d l1}-\eqref{eq: PJ-w11 eq5 2d l1}:
\begin{align}
\|v\|_{L^1(0,1)} &\le \|v\|_{L^1(0,h)} + \|v-v_h\|_{L^1(h,1)} + \|v_h\|_{L^1(h,1)} \nonumber \\
&\le \dfrac{3}{2}h\|v'\|_{L^1(0,1)} + 2\left(3-\sqrt{2}\right)\pi^{-1/2}h^{-1/2}\|g\|_{L^1(0,1)} \nonumber \\
&\le 3h\|v'\|_{L^1(0,1)} + 2\left(3-\sqrt{2}\right)\pi^{-1/2}h^{-1/2}\|g\|_{L^1(0,1)}. \label{eq: PJ-w11 eq6 2d l1}
\end{align}
Minimizing the right-hand side of Equation \eqref{eq: PJ-w11 eq6 2d l1} with the constraint $h\in(0, 1/2]$ yields: 
\begin{align}
h^* = \left(\dfrac{\left(3-\sqrt{2}\right)\|g\|_{L^1(0,1)}}{3\sqrt{\pi}\|v'\|_{L^1(0,1)}}\right)^{2/3}. \label{eq: PJ-w11 eq7 2d l1}
\end{align}
To check that the minimizer satisfies the constraint, we apply Young's inequality for convolutions to obtain:
\begin{align}
\|g\|_{L^1(0,1)} = \|K*f\|_{L^1(0,1)} \le \|K\|_{L^1(\R)}\|f\|_{L^1(\R)} = \dfrac{4}{3\sqrt{\pi}} \|v'\|_{L^1(0,1)}, \label{eq: PJ-w11 eq9 2d l1}
\end{align}
where the functions $K$ and $f$ are defined by Equation \eqref{eq: PJ-w11 kernel g}. Combining Equations \eqref{eq: PJ-w11 eq7 2d l1} and \eqref{eq: PJ-w11 eq9 2d l1} yields:
\begin{align*}
h^* = \left(\dfrac{\left(3-\sqrt{2}\right)\|g\|_{L^1(0,1)}}{3\sqrt{\pi}\|v'\|_{L^1(0,1)}}\right)^{2/3} \le \left(\dfrac{4\left(3-\sqrt{2}\right)}{9\pi}\right)^{2/3} \le \dfrac{1}{2}. 
\end{align*}
By optimizing the right-hand side of Equation \eqref{eq: PJ-w11 eq6 2d l1} with respect to $h$, we obtain the following stability estimate:
\begin{align*}
\|v\|_{L^1(0,1)} \le 3^{4/3}\pi^{-1/3}\left(3-\sqrt{2}\right)^{2/3}\|v'\|_{L^1(0,1)}^{1/3}\|g\|_{L^1(0,1)}^{2/3}.
\end{align*}
\end{proof}


\begin{theorem} \label{thm: PJ bv bound 2d l1}
If $v\in BV(0,1)$ with $\supp(v)\subset[0,1)$ and $\abelJ v = g$, we have:
\begin{align}
\|v\|_{L^1(0,1)} \le C\|v\|_{TV(0,1)}^{1/3}\|g\|_{L^1(0,1)}^{2/3}, \label{eq: PJ bv bound 2d l1}
\end{align}
where $C$ is a constant independent of $v$.
\end{theorem}
\begin{proof}
The proof is similar to the proof of Theorem \ref{thm: PJ bv bound 2d}. Using the smooth approximation theorem for $BV$ functions, there exists a sequence of functions $\{v_k\}_{k=1}^\infty\subset W^{1,1}(0,1)\cap C^\infty(0,1)=BV(0,1)\cap C^\infty(0,1)$ with the following properties: \begin{subequations} \label{eq: PJ-bv eq1 2d l1}
\begin{align}
&\|v_k-v\|_{L^1(0,1)}\to0 \quad {\rm as}\,\,k\to\infty, \label{eq: PJ-bv eq1a 2d l1} \\
& v_k\to v \quad {\rm a.e.\,\,as}\,\,k\to\infty, \label{eq: PJ-bv eq1b 2d l1} \\
{\rm and}\quad &\|v_k\|_{TV(0,1)}\to\|v\|_{TV(0,1)} \quad {\rm as}\,\,k\to\infty. \label{eq: PJ-bv eq1c 2d l1}
\end{align}
\end{subequations} Let $g_k=\abelJ v_k$. Then by Theorem \ref{thm: PJ w11 bound 2d l1},
\begin{align}
\|v_k\|_{L^1(0,1)} \le C\|v_k'\|_{L^1(0,1)}^{1/3}\|g_k\|_{L^1(0,1)}^{2/3}, \label{eq: PJ-bv eq2 2d l1}
\end{align}
where constant $C$, independent of the choice of the approximating sequence. The functions $v_k$ are $C^1(0,1)$. Therefore, condition \eqref{eq: PJ-bv eq1c 2d l1} implies that:
\begin{align}
\|v_k'\|_{L^1(0,1)}\to\|v\|_{TV(0,1)} \quad {\rm as}\,\,k\to\infty. \label{eq: PJ-bv eq3 2d l1}
\end{align}
On the other hand, choosing $p=1$ and $\epsilon=1/2$ in Theorem \ref{thm: J lp continuity 2d} so that $s=1$, and applying condition \eqref{eq: PJ-bv eq1a 2d l1}, we have:
\begin{align}
\|g_k-g\|_{L^1(0,1)} &= \|\abelJ(v_k-v)\|_{L^1(0,1)} \le \dfrac{2}{\sqrt{\pi}}\|v_k-v\|_{L^1(0,1)} \to 0 \label{eq: PJ-bv eq4 2d l1}
\end{align}
as $k\to\infty$. Therefore, by Equations \eqref{eq: PJ-bv eq2 2d l1}-\eqref{eq: PJ-bv eq4 2d l1}:
\begin{align*}
\|v\|_{L^1(0,1)} \le \liminf_{k\to\infty}\|v_k\|_{L^1(0,1)} \le C\lim_{k\to\infty}\|v_k'\|_{L^1(0,1)}^{1/3}\|g_k\|_{L^1(0,1)}^{2/3} = C\|v\|_{TV(0,1)}^{1/3}\|g\|_{L^1(0,1)}^{2/3},
\end{align*}
where the first step follows from condition \eqref{eq: PJ-bv eq1b 2d l1} and Fatou's Lemma. 
\end{proof}


By a density argument, one can obtain the following theorem from Equations \eqref{eq: PJ bv bound 2d l1} and \eqref{eq: r-bound to xy-bound 2d}.

\begin{theorem} \label{thm: PA bv bound 2d l1}
Let $u:B(0,1)\subset\R^2\to\R$ be an axisymmetric function such that, as a function of $r$, $u\in BV(0,1)$ and $\supp(u)\subset[0,1)$. If $\abel u = f$, we have:
\begin{align*}
\|u\|_{L^1(B(0,1))} \le C\|u\|_{TV(0,1)}^{1/3}\|f\|_{L^1(0,1)}^{2/3}, 
\end{align*}
where $C$ is a constant independent of $u$,
\begin{align*}
\|u\|_{L^1(B(0,1))} &:= \iint_{B(0,1)} |u(x,y)|\dd{x}\dd{y},
\end{align*}
and $\|u\|_{TV(0,1)}$ is defined by Equation \eqref{eq: def TV norm 2d}.
\end{theorem}


We now extend the preceding results to $L^1$-stability estimates for $BV$ solutions in 3D.


\begin{lemma} \label{lem: basic bounds 3d l1}
Let $v\in W^{1,1}(\Omega)$. Let $h\in(0, 1/2]$ and define $v_h$ by Equation \eqref{eq: running average 3d}. Then the following estimate holds:
\begin{align}
\|v-v_h\|_{L^1(\Omega_h)} &\le 2^{-1}h\|Dv\|_{L^1(\Omega)}. \label{eq: basic bound 3d l1}
\end{align}
\end{lemma}

\begin{proof}
Replacing $v(\cdot)$ by $v(\cdot,z)$ in the proof of Lemma \ref{lem: basic bounds 2d l1}, one can obtain the following estimate from Equation \eqref{eq: basic bound 2d l1}:
\begin{align}
\|v(\cdot,z)-v_h(\cdot,z)\|_{L^1(h,1)} \le 2^{-1}h\|D_1v(\cdot,z)\|_{L^1(0,1)}, \quad \text{a.e.}\ z\in[-1,1]. \label{eq: basic bound eq1 3d l1}
\end{align}
Integrating Equation \eqref{eq: basic bound eq1 3d l1} in $z$ from $[-1,1]$ yields:
\begin{align*}
\|v-v_h\|_{L^1(\Omega_h)} &\le 2^{-1}h\|D_1v\|_{L^1(\Omega)} \le 2^{-1}h\|Dv\|_{L^1(\Omega)},
\end{align*}
which shows Equation \eqref{eq: basic bound 3d l1}.
\end{proof}


\begin{theorem} \label{thm: PJ w11 bound 3d l1}
If $v\in W^{1,1}(\Omega)$ with $\supp(v)\subset[0,1)\times[-1,1]$ and $\abelJ v = g$, we have:
\begin{align}
\|v\|_{L^1(\Omega)} &\le C\|Dv\|_{L^1(\Omega)}^{1/3}\|g\|_{L^1(\Omega)}^{2/3}, \label{eq: PJ w11 bound 3d l1}
\end{align}
where $C$ is a constant independent of $v$.
\end{theorem}

\begin{proof}
Replacing $v(\cdot)$ by $v(\cdot,z)$ in the proof of Lemma \ref{lem: basic bounds 3d l1} yields the following estimate from Equation \eqref{eq: PJ w11 bound 2d l1}:
\begin{align}
\|v(\cdot,z)\|_{L^1(0,1)} &\le C\|D_1v(\cdot,z)\|_{L^1(0,1)}^{1/3}\|g(\cdot,z)\|_{L^1(0,1)}^{2/3}, \quad \text{a.e.}\ z\in[-1,1]. \label{eq: PJ w11 bound eq1 3d l1}
\end{align}
Integrating Equation \eqref{eq: PJ w11 bound eq1 3d l1} in $z$ from $[-1,1]$ yields:
\begin{align*}
\|v\|_{L^1(\Omega)} &\le C\|D_1v\|_{L^1(\Omega)}^{1/3}\|g\|_{L^2(\Omega)}^{2/3} \le C\|Dv\|_{L^1(\Omega)}^{1/3}\|g\|_{L^2(\Omega)}^{2/3},
\end{align*}
which shows Equation \eqref{eq: PJ w11 bound 3d l1}.
\end{proof}


\begin{theorem} \label{thm: PJ bv bound 3d l1}
If $v\in BV(\Omega)$ with $\supp(v)\subset[0,1)\times[-1,1]$ and $\abelJ v = g$, we have:
\begin{align}
\|v\|_{L^1(\Omega)} &\le C\|v\|_{TV(\Omega)}^{1/3}\|g\|_{L^1(\Omega)}^{2/3}, \label{eq: PJ bv bound 3d l1}
\end{align}
where $C$ is a constant independent of $v$.
\end{theorem}

\begin{proof}
Equation \eqref{eq: PJ bv bound 3d l1} can be derived using the same density argument as in the proof of Theorem \ref{thm: PJ bv bound 2d l1}.
\end{proof}


The following theorem is a consequence of Equations \eqref{eq: PJ bv bound 3d l1} and \eqref{eq: r-bound to xy-bound 3d}, which extends the $L^1$-stability estimate to Problem \eqref{eq: f=Au model}.


\begin{theorem} \label{thm: PA bv bound 3d l1}
Let $u:U\subset\R^3\to\R$ be an axisymmetric function such that, as a function of $(r,z)$, $u\in BV(\Omega)$ and $\supp(u)\subset[0,1)\times[-1,1]$. If $\abel u = f$, we have:
\begin{align*}
\|u\|_{L^1(U)} &\le C\|u\|_{TV(\Omega)}^{1/3}\|f\|_{L^1(\Omega)}^{2/3}, 
\end{align*}
where $C$ is a constant independent of $u$,
\begin{align*}
\|u\|_{L^1(U)} &:= \int_{-1}^1\iint_{B(0,1)} |u(x,y,z)|\dd{x}\dd{y}\dd{z}, 
\end{align*}
and $\|u\|_{TV(\Omega)}$ is defined by Equation \eqref{eq: def TV norm 3d}.
\end{theorem}


\section{Numerical Method}
\label{sec: method}

Suppose the data $f$ is 2D and is measured as a set of discrete points $\{f(x_i,z_j): i=1,\cdots,N, \,\, j=1,\cdots,M\}$ (when $M=1$, it reduces to the case where $f$ is 1D). To solve Problem \eqref{eq: TV model} numerically, we introduce the following discrete operators. 

\begin{definition} \label{def: discrete operators}
Assume that $X = \{(x_i,z_j): i=1,\cdots,N, \,\, j=1,\cdots,M\}$ is an $N\times M$ grid with grid-spacing equal to $h$. 
\begin{enumerate}[(i)]
\item 
If $u\in X$, then the discrete gradient $\nabla_h u$ of $u$ is a vector in $X\times X$ given by:
\begin{align*}
(\nabla_h u)_{i,j} = \left( (\nabla_h u)_{i,j}^1, (\nabla_h u)_{i,j}^2\right)
\end{align*}
for $i=1,\cdots,N$, $j=1,\cdots,M$, where
\begin{align*}
(\nabla_h u)_{i,j}^1 = \begin{cases}
(u_{i+1,j}-u_{i,j})/h &\quad\text{if }i<N, \\
0 &\quad\text{if }i=N,
\end{cases} \\
(\nabla_h u)_{i,j}^2 = \begin{cases}
(u_{i,j+1}-u_{i,j})/h &\quad\text{if }j<M, \\
0 &\quad\text{if }j=M;
\end{cases}
\end{align*}
see, for example, \cite{Chambolle04}.
\item 
If $p=(p^1,p^2)\in X\times X$, then the discrete divergence $\dir_h p$ of $p$ is a vector in $X$ given by:
\begin{align*}
(\dir_h p)_{i,j} = (\dir_h p)_{i,j}^1 + (\dir_h p)_{i,j}^2
\end{align*}
for $i=1,\cdots,N$, $j=1,\cdots,M$, where
\begin{align*}
(\dir_h p)_{i,j}^1 = \begin{cases}
(p_{i,j}^1-p_{i-1,j}^1)/h &\quad\text{if }1<i<N, \\
p_{i,j}^1/h &\quad\text{if }i=1, \\
-p_{i-1,j}^1/h &\quad\text{if }i=N,
\end{cases} \\
(\dir_h p)_{i,j}^2 = \begin{cases}
(p_{i,j}^2-p_{i,j-1}^2)/h &\quad\text{if }1<j<M, \\
p_{i,j}^2/h &\quad\text{if }j=1, \\
-p_{i,j-1}^2/h &\quad\text{if }j=M;
\end{cases}
\end{align*}
see, for example, \cite{Chambolle04}.
\item 
The discrete Abel transform $A: X\to X$ is a matrix of size $N\times N$, where
\begin{align*}
A_{ij} &= 
\begin{cases}
2 \left(\sqrt{x_j^2-x_i^2} - \sqrt{x_{j-1}^2-x_i^2}\right) &\quad\text{if } i<j, \\
0 &\quad\text{otherwise}
\end{cases}
\end{align*}
for $i,j=1,\cdots, N$. Derivation of $A$ is based on the onion-peeling method, see \cite{Dasch92, Pretzier92}.
\end{enumerate}
\end{definition}

\begin{remark}
One can verify using summation by parts that $(\nabla_h)^*=-\dir_h$. 
\end{remark}

Consider an axisymmetric function $u$ which is compactly supported in the cylindrical domain $U$. The Abel transform $f$ of $u$ is then compactly supported in $\Omega$. Let $U_h$ and $V_h$ be discretizations of $[-1,1]\times[-1,1]\times[0,1]\subset\R^3$ and $\Omega$, respectively:
\begin{align*}
U_h &= \left\{(x_i,y_j,z_k): -N\le i,j \le N, \,\, 0\le k\le N \right\}, \\
V_h &= \left\{(x_i,z_k): 0\le i,k \le N \right\},
\end{align*}
where $\{x_i\}_{i=-N}^N$ and $\{y_j\}_{j=-N}^N$ are the equi-spaced partition of $[-1,1]$ with grid-spacing $h$ equal to $1/N$, and $\{z_k\}_{k=0}^N$ is the equi-spaced partition of $[0,1]$ with the same grid-spacing. One can verify that if the data $f=f(x,z)$ is measured discretely on the grid $V_h$, then the information of $u=u(r,z)$ on the same grid can be obtained, and vise versa. Therefore, there is no distinction between partitioning the positive $x$-axis and partitioning the $r$-axis in the discrete setting for the Abel inverse problem.

To analyze the numerical solution, we define various discrete norms that relate to the analytical results derived in Section \ref{sec: theory}.


\begin{definition} \label{def: discrete norms}
Let $U_h$ and $V_h$ be defined as above. Let $u$ be an axisymmetric function which is evaluated discretely on the grid $U_h$ as a function of $(x,y,z)$, and on the grid $V_h$ as a function of $(r,z)$.
\begin{enumerate}[(i)]
\item The discrete $\ell^2$ norm of $u$ with respect to the Cartesian coordinates is defined by:
\begin{align*}
\|u\|_{\ell^2(U_h)} := h^{3/2}\sum_{i,j=-N}^N\sum_{k=1}^N u_{i,j,k}^2,
\end{align*}
where $u_{i,j,k} = u(x_i,y_j,z_k)$.
\item The discrete $\ell^2$ norm of $u$ with respect to the cylindrical coordinates is defined by:
\begin{align*}
\|u\|_{\ell^2(V_h)} := h\sum_{i=-N}^N\sum_{k=1}^N u_{i,k}^2,
\end{align*}
where $u_{i,k} = u(r_i,z_k)$.
\item The discrete $BV$ semi-norm of $u$ with respect to the cylindrical coordinates is defined by:
\begin{align*}
\|\nabla_h u\|_{\ell^1(V_h\times V_h)} := h^2 \sum_{i,j=1}^N\left|(\nabla_h u)_{i,j}\right| = h^2 \sum_{i,j=1}^N \sqrt{ \left((\nabla_h u)_{i,j}^1\right)^2 + \left((\nabla_h u)_{i,j}^2\right)^2 }.
\end{align*}
\item The discrete $\ell^\infty$ norm of $u$ is defined by:
\begin{align*}
\|u\|_{\ell^\infty(V_h)} := \max_{i,k=1,\cdots,N} |u_{i,k}|,
\end{align*}
where $u_{i,k} = u(r_i,z_k)$. This quantity is independent of the choice of coordinate system.
\end{enumerate}
\end{definition}

The primal-dual algorithm \cite{Chambolle11} applied to Problem \eqref{eq: discrete TV model} is summarized in Algorithm \ref{alg: primal dual}. The output $u^*$ of Algorithm \ref{alg: primal dual} is a discrete approximation to the solution $u=u(r,z)$ of Problem \eqref{eq: TV model}. The following theorem shows that the convergence of the primal-dual algorithm applied to Problem \eqref{eq: discrete TV model} is $\mathcal{O}(1/n)$, where $n$ is the number of iterations.

\begin{algorithm} [t!]
\caption{The primal-dual algorithm applied to Problem \eqref{eq: discrete TV model}}
\label{alg: primal dual}
\begin{algorithmic}[1]
\State Choose $\tau,\gamma>0$. Initialize $u^0\in V_h$ and $v^0\in V_h\times V_h$. Set $w^0 = u^0$ and $n=0$. Let \texttt{MaxIter} be the maximum number of iterations allowed.
\While{$n\le \texttt{MaxIter}$}
\State $p^n = v^n+\gamma\nabla_h w^n$
\State $v^{n+1} = p^n/\max(1,|p^n|)$, where the operation is preformed component-wise 
\State $q^n = u^n + \tau \dir_h v^{n+1}$
\State $u^{n+1} = \left(I + \tau\lambda A^TA\right)^{-1}\left(q^n+\tau\lambda A^Tf\right)$
\State $w^{n+1} = 2u^{n+1}-u^n$
\EndWhile
\State \Return $u^* = u^{n+1}$
\end{algorithmic}
\end{algorithm}

\begin{theorem} \label{Chambolle11 thm 1}
{\rm (restated from \cite{Chambolle11})}
Consider the sequence $(u^n,v^n)$ defined by Algorithm \ref{alg: primal dual} and let $(u^*,v^*)$ be the unique solution of the corresponding saddle-point form of Problem \eqref{eq: discrete TV model}:
\begin{align*}
\min_{u\in V_h}\max_{v\in V_h\times V_h} \quad & \langle \nabla_h u, v\rangle_{V_h\times V_h} + \|Au-f\|_{\ell^2(V_h)}^2 - \chi_B(v),
\end{align*}
where $\chi_B$ is the characteristic function of the unit ball $B$ in $\ell^\infty(V_h\times V_h)$:
\begin{align*}
\chi_B(v) =
\begin{cases}
0 \quad & \text{if } \|v\|_{\ell^\infty(V_h\times V_h)}\le 1, \\
+\infty \quad & \text{otherwise}.
\end{cases}
\end{align*}
Then $||u^n-u^*||_{\ell^2(V_h)}=\mathcal{O}(1/n)$. 
\end{theorem}


\section{Auxiliary Results}
\label{sec: appendix}

To be self-contained, we include some results that we used in the main text.


\begin{proof}[Proof of Theorem \ref{thm: PJ uniqueness 2d}]
The arguments below are adapted from the proof of Theorem 1.A.1 in \cite{Gorenflo91}, where we have modified some calculations to fit our context. We will first show the existence of a solution and then the uniqueness.

Let $v$ be a function defined by Equation \eqref{eq: PJ solution 2d} and by Fubini's theorem:
\begin{align*}
\abelJ v(x) = \dfrac{1}{\sqrt{\pi}}\int_x^1 \dfrac{v(r)}{\sqrt{r-x}}\dd{r} &= -\dfrac{1}{\pi}\int_x^1 \dfrac{1}{\sqrt{r-x}} \int_r^1 \dfrac{\dd{g(y)}}{\sqrt{y-r}} \dd{r} \\
&= -\dfrac{1}{\pi}\int_x^1 \int_x^y \dfrac{1}{\sqrt{r-x}\sqrt{y-r}} \dd{r}\dd{g(y)} \\
&= -\int_x^1\dd{g(y)} = g(x),
\end{align*}
where we have used the identity \cite{Epstein08, Whittaker27}:
\begin{align}
\int_x^y (r-x)^{-1/2}(y-r)^{-1/2}\dd{r} = \pi \label{eq: gamma identity}
\end{align}
and the assumption that $\supp(g)\subset[0,1)$. Therefore, $v$ is a solution to Problem \eqref{eq: g=Jv model}. 

We now show that $v\in L^1(0,1)$. Decompose $g$ to be $g_1-g_2$, where $g_1$ and $g_2$ are two bounded decreasing functions such that $g_i(0)\ge0$ and $\supp(g_i)\subset[0,1)$, $i=1,2$. Such a decomposition is guaranteed by, for example, Theorem 3.27 in \cite{Folland99}. Therefore, $\dd{g} = \dd{g_1}-\dd{g_2}$, and 
\begin{align*}
v(r) = -\dfrac{1}{\sqrt{\pi}}\int_r^1 \dfrac{\dd{g_1(x)}}{\sqrt{x-r}} + \dfrac{1}{\sqrt{\pi}}\int_r^1 \dfrac{\dd{g_2(x)}}{\sqrt{x-r}}.
\end{align*}
By triangle inequality,
\begin{align*}
\int_0^1|v(r)|\dd{r} \le -\dfrac{1}{\sqrt{\pi}}\int_0^1\int_r^1 \dfrac{\dd{g_1(x)}}{\sqrt{x-r}}\dd{r} - \dfrac{1}{\sqrt{\pi}}\int_0^1\int_r^1 \dfrac{\dd{g_2(x)}}{\sqrt{x-r}},
\end{align*}
where the two minus signs on the right-hand side come from the fact that $g_1$ and $g_2$ are decreasing functions. For $i=1,2$, we have
\begin{align}
-\int_0^1 \int_r^1 \dfrac{\dd{g_i(x)}}{\sqrt{x-r}}\dd{r} =  -2\int_0^1 \sqrt{x}\dd{g_i(x)} \le -2\int_0^1 \dd{g_i(x)} = 2g_i(0) < \infty. \label{eq: PJ uniqueness eq1 2d}
\end{align}
Therefore, $v\in L^1(0,1)$.

We now prove the uniqueness of solutions. Let $v\in L^1(0,1)$ be in the null space of $\abelJ$, {\it i.e.}
\begin{align}
\dfrac{1}{\sqrt{\pi}} \int_x^1 \dfrac{v(r)}{\sqrt{r-x}}\dd{r} = 0, \quad x\in[0,1]. \label{eq: PJ uniqueness eq2 2d}
\end{align}
Choosing a $y\in[0,1]$ and using Fubini's theorem with Equations \eqref{eq: gamma identity} and \eqref{eq: PJ uniqueness eq2 2d}, we have:
\begin{align*}
0 &= \int_y^1\left( \dfrac{1}{\sqrt{\pi}} \int_x^1 \dfrac{v(r)}{\sqrt{r-x}}\dd{r}\right)\,\dfrac{1}{\sqrt{\pi}\sqrt{x-y}}\dd{x}, \\
&= \int_y^1\left( \dfrac{v(r)}{\pi} \int_y^r \dfrac{1}{\sqrt{r-x}\sqrt{x-y}}\,\dd{x}\right)\dd{r} = \int_y^1 v(r)\dd{r}.
\end{align*}
Since $\int_y^1 v(r)\dd{r}=0$ for all $y\in[0,1]$, by Lebesgue differentiation theorem, we have $v=0$ almost everywhere on $[0,1]$. This shows that Problem \eqref{eq: g=Jv model} has a unique solution.
\end{proof}


To relate various semi-norms and norms in Cartesian and cylindrical coordinates, we have the following two propositions.


\begin{proposition} \label{prop: TV norm cartesian}
If $u:B(0,1)\subset\R^2\to\R$ is an axisymmetric function and is of bounded variation in $B(0,1)$, then 
\begin{align*}
\|u\|_{TV(B(0,1))} = 2\pi \|u\|_{TV(0,1),r},
\end{align*}
where
\begin{align*}
\|u\|_{TV(B(0,1))} &:=  \sup\left\{ \iint_{B(0,1)} u(x,y) \dir\phi(x,y)\dd{x}\dd{y}: \phi\in C_c^1(B(0,1);\R^2), \, \|\phi\|_{L^\infty(B(0,1))}\le1 \right\}, \\
\|u\|_{TV(0,1),r} &:=  \sup\left\{ \int_0^1 u(r) \dir(r\phi(r))\dd{r}: \phi\in C_c^1((0,1);\R^2), \, \|\phi\|_{L^\infty(0,1)}\le1 \right\}.
\end{align*}
\end{proposition}

\begin{proof}
Let $u$ be a $C^1$ axisymmetric function with $u(r,\theta)=u(r)$. One can show that
\begin{align*}
\dfrac{\partial u}{\partial x} = \dfrac{\partial u}{\partial r}\cos\theta, \quad \dfrac{\partial u}{\partial y} =  \dfrac{\partial u}{\partial r}\sin\theta,
\end{align*}
which implies that $u_x^2 + u_y^2 = u_r^2$. Therefore,
\begin{align*}
\iint_{B(0,1)}|\nabla u(x,y)|\dd{x}\dd{y} &= \iint_{B(0,1)}\left|\left(u_x(x,y),u_y(x,y)\right)\right|\dd{x}\dd{y} \\
&= \int_0^{2\pi}\int_0^1|u_r(r,\theta)|r\dd{r}\dd{\theta} = 2\pi \int_0^1|u'(r)|r\dd{r}.
\end{align*}
The extension to $BV$ functions can be concluded from a density argument.
\end{proof}


The following three results provide information about the continuity of the $\abelJ$-transform.


\begin{proposition} \label{prop: A-J norm}
Assume that $u:\R^2\to\R$ is an axisymmetric function which is compactly supported within the ball $B(0,1)\subset\R^2$, and that $v:\R^+\to\R$ is the function such that $v(r^2) = u(r)$. Let $f := \abel u$ and $g:=\abelJ v$. If $u$ and $v$ are smooth, then 
\begin{subequations}
\begin{align}
\|v\|_{L^2(0,1)}^2 & = \dfrac{1}{\pi}\|u\|_{L^2(B(0,1))}^2, \\
\|v'\|_{L^1(0,1)} &= \|u'\|_{L^1(0,1)}, \\
\|g\|_{L^2(0,1)}^2 &\le 2\|f\|_{L^2(0,1)}^2.
\end{align} \label{eq: r-bound to xy-bound 2d}
\end{subequations}
Similarly, assume that $u:\R^3\to\R$ is an axisymmetric function which is compactly supported in the cylinder $U\subset\R^3$, and that $v:\R^+\times\R\to\R$ is the function such that $v(r^2,z) = u(r,z)$. Let $f = \abel u$ and $g=\abelJ v$. If $u$ and $v$ are smooth, then
\begin{subequations}
\begin{align}
\|v\|_{L^2(\Omega)}^2 &= \dfrac{1}{\pi}\|u\|_{L^2(U)}^2, \\
\|Dv\|_{L^1(\Omega)} &\le 2\|Du\|_{L^1(\Omega)}, \\
\|g\|_{L^2(\Omega)}^2 &\le 2\|f\|_{L^2(\Omega)}^2.
\end{align} \label{eq: r-bound to xy-bound 3d}
\end{subequations}
\end{proposition}

\begin{proof}
Equation \eqref{eq: r-bound to xy-bound 2d} is a consequence of a change-of-variable and the chain rule: 
\begin{align*}
\|v\|_{L^2(0,1)}^2 &= \int_0^1 \left|u(\sqrt{r})\right|^2\dd{r} = \int_0^1 2r|u(r)|^2\dd{r} \\
&= \dfrac{1}{\pi}\iint_{B(0,1)} |u(x,y)|^2\dd{x}\dd{y} = \dfrac{1}{\pi}\|u\|_{L^2(B(0,1))}^2, \\
\|v'\|_{L^1(0,1)} &= \int_0^1 \left|\dfrac{\dd{u(\sqrt{r})}}{\dd{r}}\right|\dd{r} = \int_0^1 \left|\dfrac{\dd{u(\sqrt{r})}}{\dd{\sqrt{r}}}\times\dfrac{\dd{\sqrt{r}}}{\dd{r}}\right|\dd{r} \\
&= \int_0^1 \left|\dfrac{u'(s)}{2s}\right|\dd{s^2} = \int_0^1 |u'(r)|\dd{r} = \|u'\|_{L^1(0,1)}, \\
\|g\|_{L^2(0,1)}^2 &= \int_0^1 \left(\dfrac{1}{\sqrt{\pi}} \int_x^1 \dfrac{u(\sqrt{r})}{\sqrt{r-x}}\dd{r}\right)^2 \dd{x} = \int_0^1 \left(\dfrac{2}{\sqrt{\pi}} \int_{\sqrt{x}}^1 \dfrac{u(r)r}{\sqrt{r^2-x}}\dd{r}\right)^2 \dd{x} \\
&= \int_0^1 2x\left(\dfrac{2}{\sqrt{\pi}} \int_x^1 \dfrac{u(r)r}{\sqrt{r^2-x^2}}\dd{r}\right)^2 \dd{x} = \int_0^1 2x\left|f(x)\right|^2 \dd{x} \le 2\|f\|_{L^2(0,1)}^2.
\end{align*}

Equation \eqref{eq: r-bound to xy-bound 3d} can be obtained from Equation \eqref{eq: r-bound to xy-bound 2d} using the same calculation as above:
\begin{align*}
\|v\|_{L^2(\Omega)}^2 &= \int_\Omega \left|u(\sqrt{r},z)\right|^2\dd{r}\dd{z} = \dfrac{1}{\pi}\|u\|_{L^2(U)}^2, \\
\|Dv\|_{L^1(\Omega)} &= \int_{-1}^1\int_0^1 \sqrt{v_r(r,z)^2+v_z(r,z)^2} \dd{r}\dd{z} \\
&= \int_{-1}^1\int_0^1 \sqrt{u_r(r,z)^2+4r^2u_z(r,z)^2} \dd{r}\dd{z} \\
&\le 2 \int_{-1}^1\int_0^1 \sqrt{u_r(r,z)^2+u_z(r,z)^2} \dd{r}\dd{z} = 2\|Du\|_{L^1(\Omega)}, \\
\|g\|_{L^2(\Omega)}^2 &= \int_{-1}^1\int_0^1 2x\left|f(x,z)\right|^2 \dd{x}\dd{z} \le 2\|f\|_{L^1(\Omega)}^2.
\end{align*}
\end{proof}


\begin{theorem} \label{thm: J lp continuity 2d}
{\rm (a special case of Theorem 4.1.1 in \cite{Gorenflo91})}
If $v\in L^p(0,1)$, $1 \le p \le 2$, and $s=p(1-p(1/2-\epsilon))^{-1}$ with $\epsilon>0$, then
\begin{align*}
\|\abelJ v\|_{L^s(0,1)} \le \dfrac{1}{\sqrt{\pi}}\left(1+\dfrac{1}{2\epsilon}\right)^{1/2+\epsilon}\|v\|_{L^p(0,1)}. 
\end{align*}
\end{theorem}
\begin{proof}
For the sake of completeness, we provide a proof which is skipped in \cite{Gorenflo91}.

We have:
\begin{align*}
\|\abelJ v\|_{L^s(0,1)} &= \left(\int_0^1 \left(\abelJ v(x)\right)^s \,{\rm d}x\right)^{1/s} = \dfrac{1}{\sqrt{\pi}}\left(\int_0^1 \left( \int_x^1 \dfrac{v(r)}{\sqrt{r-x}}\,{\rm d}r \right)^s \,{\rm d}x\right)^{1/s} \\
&= \dfrac{1}{\sqrt{\pi}}\left(\int_0^1 \left( K*f(x) \right)^s \,{\rm d}x\right)^{1/s} = \dfrac{1}{\sqrt{\pi}}\|K*f\|_{L^s(0,1)},
\end{align*}
where
\begin{align*}
K(x) = \dfrac{1}{\sqrt{|x|}}\mathds{1}_{[0,1]}(-x), \quad f(x) = v(x)\mathds{1}_{[0,1]}(x).
\end{align*}
Therefore, as a consequence of Young's inequality for convolutions, with $K$ and $f$ extended to $\R$,
\begin{align*}
\|\abelJ v\|_{L^s(0,1)} = \dfrac{1}{\sqrt{\pi}}\|K*f\|_{L^s(0,1)} \le \dfrac{1}{\sqrt{\pi}}\|K*f\|_{L^s(\R)} \le \dfrac{1}{\sqrt{\pi}}\|K\|_{L^q(\R)}\|f\|_{L^p(\R)},
\end{align*}
where $q$ solves $s^{-1} = p^{-1}+q^{-1}-1$,
{\it i.e.} $q=(1/2+\epsilon)^{-1}$. The $L^q$ norm of $K$ is equal to $\left(1+1/(2\epsilon)\right)^{1/2+\epsilon}$. The $L^p$ norm of $f$ is equal to $\|v\|_{L^p(0,1)}$. 
Thus,
\begin{align*}
\|\abelJ v\|_{L^s(0,1)} \le \dfrac{1}{\sqrt{\pi}}\left(1+\dfrac{1}{2\epsilon}\right)^{1/2+\epsilon}\|v\|_{L^p(0,1)}.
\end{align*}
\end{proof}


\begin{corollary} \label{cor: J lp continuity 3d}
If $v\in L^p(\Omega)$, $1 \le p \le 2$, then
\begin{align*}
\|\abelJ v\|_{L^p(\Omega)} \le \dfrac{2}{\sqrt{\pi}}\|v\|_{L^p(\Omega)}. 
\end{align*}
\end{corollary}


\begin{proposition} \label{prop: J linf continuity 2d}
{\rm (adapted from \cite{Samko93})}
The operator $\abelJ$ defined in Equation \eqref{eq: def abelJ 2d} is a continuous operator from $L^\infty(0,1)$ into $C^{0,1/2}(0,1)$.
\end{proposition}

\begin{proof}
The arguments below are adapted from the proof of Corollary 2 on page 56 in \cite{Samko93}, where we have modified some calculations to fit our context. 

Let $v\in L^\infty(0,1)$. Fix $x$ and $h$ such that $0\le x< x+h\le1$. By triangle inequality,
\begin{align*}
\sqrt{\pi}|\abelJ v(x+h)-\abelJ v(x)| &= \left| \int_{x+h}^1 \dfrac{v(r)}{\sqrt{r-x-h}}\,{\rm d}r - \int_x^1 \dfrac{v(r)}{\sqrt{r-x}}\,{\rm d}r\right| \\
&= \left| \int_{x+h}^1 \dfrac{v(r)}{\sqrt{r-x-h}}-\dfrac{v(r)}{\sqrt{r-x}}\,{\rm d}r - \int_x^{x+h} \dfrac{v(r)}{\sqrt{r-x}}\,{\rm d}r \right| \\
&\le \left( \int_{x+h}^1 \dfrac{1}{\sqrt{r-x-h}}-\dfrac{1}{\sqrt{r-x}}\,{\rm d}r + \int_x^{x+h} \dfrac{1}{\sqrt{r-x}}\,{\rm d}r \right)\|v\|_{L^\infty(0,1)} \\
&= \left( 4\sqrt{h} + 2\sqrt{1-x-h} - 2\sqrt{1-x} \right)\|v\|_{L^\infty(0,1)} \\
&\le 4\sqrt{h}\|v\|_{L^\infty(0,1)},
\end{align*}
whence $|\abelJ v|_{C^{0,1/2}(0,1)} \le 4\pi^{-1/2}\|v\|_{L^\infty(0,1)}$. On the other hand,
\begin{align*}
\sqrt{\pi}|\abelJ v(x)| \le \left( \int_x^1 \dfrac{1}{\sqrt{r-x}}\,{\rm d}r \right)\|v\|_{L^\infty(0,1)} = 2\sqrt{1-x}\|v\|_{L^\infty(0,1)} \le 2\|v\|_{L^\infty(0,1)},
\end{align*}
and thus $\|\abelJ v\|_{L^\infty(0,1)}\le 2\pi^{-1/2}\|v\|_{L^\infty(0,1)}$. Therefore, 
\begin{align*}
\|\abelJ v\|_{C^{0,1/2}(0,1)} = |\abelJ v|_{C^{0,1/2}(0,1)} + \|\abelJ v\|_{L^\infty(0,1)} \le 6\pi^{-1/2}\|v\|_{L^\infty(0,1)}.
\end{align*}
This completes the proof.
\end{proof}


The remaining results provide some $BV$ estimates used in Section \ref{sec: theory}.


\begin{lemma} \label{lem: partial total variation}
{\rm (restated from \cite{Evans92, Ambrosio00, Bergounioux10})} 
Let $v\in BV(\Omega)$. For almost every $z\in[-1,1]$, the marginal function $v^z: r\to v(r, z)$ is of bounded variation on $[0,1]$. Moreover, 
\begin{align*}
\int_{-1}^1\|v(\cdot,z)\|_{TV(0,1)}\dd{z} \le \|v\|_{TV(\Omega)}.
\end{align*}
\end{lemma}


\begin{theorem} \label{thm: approximating sequence}
{\rm (adapted from \cite{Evans92})} 
Assume $f\in BV(U)\cap L^\infty(U)$, where $U$ is a bounded open subset of $\R^n$. Given $p\in[1,\infty)$, there exists a sequence $\{f_k\}_{k=1}^\infty\subset Lip(U)$ such that
\begin{enumerate}[(i)]
\item \label{list: approximating sequence cond1} $f_k\to f$ in $L^p(U)$ as $k\to\infty$,
\item \label{list: approximating sequence cond2} $\|f_k\|_{TV(U)}\to\|f\|_{TV(U)}$ as $k\to\infty$, and
\item \label{list: approximating sequence cond3} $\|f_k\|_{L^\infty(U)}\le\|f\|_{L^\infty(U)}$ for all $k$.
\end{enumerate}
\end{theorem}

Here $Lip(U)$ denotes the space of all functions on $U$ which are Lipschitz continuous on $U$.

\begin{proof}
The arguments below are adapted from the proof of Theorem 2 on page 172 in \cite{Evans92}. In particular, we want to construct an approximating sequence which is uniformly bounded in $L^\infty$ by $\|f\|_{L^\infty}$.

We start with the same construction as in \cite{Evans92}. Fix $\epsilon>0$, and define the open sets:
\begin{align*}
U_0 := \emptyset, \quad U_k:= \left\{ x\in U :  \dist(x,\partial U)>\dfrac{1}{m+k} \text{ and } \dist(x,0)\le\dfrac{1}{m+k}\right\}, \quad k\ge1, 
\end{align*}
where $m$ is a positive integer chosen sufficiently large such that:
\begin{align}
\|f\|_{TV(U\backslash U_1)} < \epsilon. \label{eq: approximating sequence eq1}
\end{align} 
Let $\{\zeta_k\}_{k=1}^\infty$ be a sequence of functions such that $\zeta_k \in C_c^\infty(V_k)$, $0\le\zeta_k\le1$, $k\ge1$, and
\begin{align}
\sum_{k=1}^\infty\zeta_k = 1 \quad {\rm on}\,\,U, \label{eq: approximating sequence eq2}
\end{align}
where
\begin{align}
V_k := U_{k+1}\backslash\overline{U}_{k-1}, \quad k\ge1. \label{eq: approximating sequence eq3}
\end{align}
Let $\eta$ be the standard mollifier. For each $k$, choose an $\epsilon_k>0$ sufficiently small such that:
\begin{subequations}
\begin{align}
&\supp(\eta_{\epsilon_k}*(f\zeta_k)) \subset V_k, \label{eq: approximating sequence eq4a} \\
&\|\eta_{\epsilon_k}*(f\zeta_k))-f\zeta_k\|_{L^p(U)} < \epsilon2^{-k}, \label{eq: approximating sequence eq4b} \\
&\|\eta_{\epsilon_k}*(fD\zeta_k))-fD\zeta_k\|_{L^p(U)} < \epsilon2^{-k}, \label{eq: approximating sequence eq4c}
\end{align}
\end{subequations}
The existence of such $\epsilon_k$ is guaranteed by the density of $\lip(U)$ in $L^p(U)$. Define
\begin{align}
f_\epsilon := \sum_{k=1}^\infty \eta_{\epsilon_k}*(f\zeta_k), \quad \tilde{f}_\epsilon := \max\{f_\epsilon,\|f\|_{L^\infty(U)}\}. \label{eq: approximating sequence eq5}
\end{align}
By Equation \eqref{eq: approximating sequence eq4a}, the sum $\sum_{k=1}^\infty \eta_{\epsilon_k}*(f\zeta_k)$ has finitely many nonzero terms when evaluated at each $x\in U$. Thus, $f_\epsilon \in C^\infty(U)$ and $\tilde{f}_\epsilon\in\lip(U)$. It can be seen immediately from Equation \eqref{eq: approximating sequence eq5} that $\|\tilde{f}_\epsilon\|_{L^\infty(U)}\le\|f\|_{L^\infty(U)}$ for all $\epsilon>0$, so that any subsequence $\{\tilde{f}_{\epsilon_k}\}_{k=1}^\infty$ of the family $\{\tilde{f}_\epsilon\}_{\epsilon>0}$ will satisfy condition \eqref{list: approximating sequence cond3}. We now show that the sequence $\{\tilde{f}_{\epsilon_k}\}_{k=1}^\infty$ can be chosen to satisfy conditions \eqref{list: approximating sequence cond1} and \eqref{list: approximating sequence cond2}. 

By partition of unity, it follows from Equations \eqref{eq: approximating sequence eq4b} and \eqref{eq: approximating sequence eq5} that:
\begin{align*}
\|\tilde{f}_\epsilon-f\|_{L^p(U)} \le \|f_\epsilon-f\|_{L^p(U)} \le \sum_{k=1}^\infty \|\eta_{\epsilon_k}*(f\zeta_k))-f\zeta_k\|_{L^p(U)} < \epsilon.
\end{align*}
Thus, $\tilde{f}_\epsilon\to f$ in $L^p(U)$ as $\epsilon\to0$, which proves condition \eqref{list: approximating sequence cond1}. 

By the $L^p$ embedding theorem:
\begin{align*}
\|\tilde{f}_\epsilon-f\|_{L^1(U)} \le C\|\tilde{f}_\epsilon-f\|_{L^p(U)}<C\epsilon,
\end{align*}
where $C$ is a constant depending only on $U$ and $p$. Thus, $\tilde{f}_\epsilon\to f$ in $L^1(U)$ as $\epsilon\to0$, and by the lower semicontinuity property of total variation:
\begin{align*}
\|f\|_{TV(U)} \le \liminf_{\epsilon\to0} \|\tilde{f}_\epsilon\|_{TV(U)}. 
\end{align*}

Following \cite{Evans92}, we now show that:
\begin{align}
\limsup_{\epsilon\to0} \|\tilde{f}_\epsilon\|_{TV(U)} \le \|f\|_{TV(U)} \label{eq: approximating sequence eq6}
\end{align}
to complete the proof for condition \eqref{list: approximating sequence cond2}. Let $\phi\in C_c^1(U;\R^n)$ with $\|\phi\|_{L^\infty(U)}\le1$. Let $\tilde{U}\subset U$ be the set such that $\tilde{f}_\epsilon=f_\epsilon$ on $\tilde{U}$ and $\tilde{f}_\epsilon=\|f\|_{L^\infty(U)}$ on $U\backslash\tilde{U}$. Since $\tilde{f}$ is constant outside $\tilde{U}$, we have
\begin{align}
\int_U \tilde{f}_\epsilon\dir(\phi)\,{\rm d}x = \int_{\tilde{U}} f_\epsilon\dir(\phi)\,{\rm d}x. \label{eq: approximating sequence eq7} 
\end{align}
For $k\geq1$, by Fubini's theorem, we have:
\begin{align}
\int_{\tilde{U}} \eta_{\epsilon_k}*(f\zeta_k)\dir(\phi)\,{\rm d}x &= \int_{\tilde{U}} \int_{\tilde{U}} \eta_{\epsilon_k}(x-y)f(y)\zeta_k(y)\dir(\phi(x))\,{\rm d}y \,{\rm d}x \nonumber \\
&= \int_{\tilde{U}} \int_{\tilde{U}} \eta_{\epsilon_k}(y-x)f(y)\zeta_k(y)\dir(\phi(x))\,{\rm d}x \,{\rm d}y \nonumber \\
&= \int_{\tilde{U}} (f\zeta_k)\eta_{\epsilon_k}*\dir(\phi)\,{\rm d}x, \label{eq: approximating sequence eq8} 
\end{align}
where the second step follows from the symmetry of $\eta_{\epsilon_k}$. Then applying the convolution-derivative theorem and the product rule, one can obtain:
\begin{align}
\int_{\tilde{U}} (f\zeta_k)\eta_{\epsilon_k}*\dir(\phi)\,{\rm d}x 
&= \int_{\tilde{U}} (f\zeta_k)\dir(\eta_{\epsilon_k}*\phi)\,{\rm d}x \nonumber \\
&= \int_{\tilde{U}} f\dir(\zeta_k(\eta_{\epsilon_k}*\phi))\,{\rm d}x - \int_{\tilde{U}} fD\zeta_k\cdot(\eta_{\epsilon_k}*\phi)\,{\rm d}x. \label{eq: approximating sequence eq9} 
\end{align}
Using the same calculation as in Equation \eqref{eq: approximating sequence eq8}, one can show that:
\begin{align}
\int_{\tilde{U}} fD\zeta_k\cdot(\eta_{\epsilon_k}*\phi)\,{\rm d}x = \int_{\tilde{U}} \phi\cdot\left( \eta_{\epsilon_k}*(fD\zeta_k)\right)\,{\rm d}x. \label{eq: approximating sequence eq10} 
\end{align}
Therefore, combining Equations \eqref{eq: approximating sequence eq5}-\eqref{eq: approximating sequence eq10} yields:
\begin{align*}
\int_U f_\epsilon\dir(\phi)\,{\rm d}x &= \sum_{k=1}^\infty \int_{\tilde{U}} f\dir(\zeta_k(\eta_{\epsilon_k}*\phi))\,{\rm d}x - \sum_{k=1}^\infty \int_{\tilde{U}} \phi\cdot\left( \eta_{\epsilon_k}*(fD\zeta_k)\right)\,{\rm d}x =: I_{1,\epsilon} + I_{2,\epsilon}.
\end{align*}
For $k\ge1$, $|\zeta_k(\eta_{\epsilon_k}*\phi)|\le1$ on $U$, and by Equation \eqref{eq: approximating sequence eq3}, each point in $U$ belongs to at most three of the sets $\{V_k\}_{k=1}^\infty$. Thus, 
\begin{align*}
|I_{1,\epsilon}| &= \left| \int_{\tilde{U}} f\dir(\zeta_1(\eta_{\epsilon_1}*\phi))\,{\rm d}x + \sum_{k=2}^\infty \int_{\tilde{U}} f\dir(\zeta_k(\eta_{\epsilon_k}*\phi))\,{\rm d}x\right| \\
&\le \|f\|_{TV(\tilde{U})} + \sum_{k=2}^\infty\|f\|_{TV(V_k)} \le \|f\|_{TV(U)} + 3\|f\|_{TV(U\backslash U_1)} < \|f\|_{TV(U)} + 3\epsilon,
\end{align*}
where the last step follows from Equation \eqref{eq: approximating sequence eq1}. Equation \eqref{eq: approximating sequence eq2} implies that $\sum_{k=1}^\infty D\zeta_k=0$ on $U$. Thus,
\begin{align*}
I_{2,\epsilon} = - \sum_{k=1}^\infty \int_{\tilde{U}} \phi\cdot\left( \eta_{\epsilon_k}*(fD\zeta_k)-fD\zeta_k\right)\,{\rm d}x,
\end{align*}
and by Equation \eqref{eq: approximating sequence eq4c}, $|I_{2,\epsilon}|<\epsilon$. Therefore,
\begin{align*}
\int_U \tilde{f}_\epsilon\dir(\phi)\,{\rm d}x < \|f\|_{TV(U)} + 4\epsilon,
\end{align*}
and
\begin{align*}
\|\tilde{f}_\epsilon\|_{TV(U)} \le \|f\|_{TV(U)} + 4\epsilon,
\end{align*}
which implies Equation \eqref{eq: approximating sequence eq6}. The proof is then complete.
\end{proof}


\begin{remark} \label{rem: approximating sequence}
In Theorem 2 on page 172 of \cite{Evans92}, a $C^\infty$ approximating sequence is constructed for $BV$ functions. For our arguments, a Lipschitz approximating sequence is sufficient in order to have the additional $L^\infty$ control.
\end{remark}


\section*{Acknowledgements}

L.Z. and H.S. acknowledge the support of AFOSR, FA9550-17-1-0125. 


\bibliographystyle{plain}
\bibliography{Paper17_Abel_references}

\end{document}